\documentclass[12pt,reqno]{amsart}
\usepackage{amsmath,amssymb,amsthm}
 \usepackage[vlined,ruled]{algorithm2e}
\usepackage{xcolor}
\usepackage{comment}
\usepackage{fullpage}

\usepackage{appendix}
\usepackage{hyperref, xcolor}
\usepackage{enumerate}

\newtheorem{thm}{Theorem}[section]
\newtheorem{lem}[thm]{Lemma}
\newtheorem{prop}[thm]{Proposition}

\newtheorem{cor}[thm]{Corollary}

\usepackage{tabstackengine}
\stackMath

\newcommand{\PP}{\mathcal{P}}
\newcommand{\QQ}{\mathcal{Q}}
\newcommand{\R}{\mathbb{R}} 
 
\newcommand{\N}{\mathbb{N}}
\newcommand{\Z}{\mathbb{Z}}
\newcommand{\F}{\mathbb{F}}

\DeclareMathOperator*{\Sum}{\mathlarger{\sum}}

\begin{document}

\title{On the directions determined by Cartesian products and the clique number of generalized Paley graphs}

\author{Chi Hoi Yip}
\address{Department of Mathematics \\ University of British Columbia \\ 1984 Mathematics Road \\ Canada V6T 1Z2}
\email{kyleyip@math.ubc.ca}

\date{\today}

\maketitle

\begin{abstract}
It is known that the number of directions formed by a Cartesian product $A \times B \subset AG(2,p)$ is at least $|A||B| - \min\{|A|,|B|\} + 2$, provided $p$ is prime and $|A||B|<p$. This implies the best known upper bound on the clique number of the Paley graph over $\mathbb{F}_p$. In this paper, we extend this result to $AG(2,q)$, where $q$ is a prime power. We also give improved upper bounds on the clique number of generalized Paley graphs over $\mathbb{F}_q$. In particular, for a cubic Paley graph, we improve the trivial upper bound $\sqrt{q}$ to $0.769\sqrt{q}+1$. In general, as an application of our key result on the number of directions, for any positive function $h$ such that $h(x)=o(x)$ as $x \to \infty$, we improve the trivial upper bound $\sqrt{q}$ to $\sqrt{q}-h(p)$ for almost all non-squares $q$.
\end{abstract}

\section{Introduction}
Let $q=p^s$ to be a prime power, and let $\F_q$ be the finite field with $q$ elements. Throughout this work, all polynomials considered will be defined over $\mathbb{F}_q$. 

In the first half of the paper, we will improve lower bounds on the number of directions formed by a Cartesian product in the affine Galois plane $AG(2,q)$, which extends the work of Di Benedetto,  Solymosi, and White \cite{BSW}.  In the second part of the paper, we will improve upper bounds on the clique number of generalized Paley graphs over $\F_q$, which extends the work of Bachoc, Matolcsi, and Ruzsa \cite{BMR}; Hanson and Petridis \cite{HP}; and Yip \cite{Yip}. The connection between the number of directions and the clique number will be made precise in Section 1.4.

\subsection{Directions determined by a point set in an affine Galois plane}
Let $AG(2,q)$ denote the {\em affine Galois plane} over the finite field $\F_q$. Let $U \subset AG(2,q)$, we use Cartesian coordinates in $AG(2,q)$ so that $U=\{(x_i,y_i):1 \leq i \leq |U|\}$.
The set of {\em directions} determined by $U \subset AG(2, \F_q)$ is 
\[ D:=D(U) = \left\{ \frac{y_j-y_i}{x_j-x_i} \colon 1\leq i <j \leq |U| \right \} \subset \F_q \cup \{\infty\}.\] 
The possible values on $|D|$ have been studied by many authors. For a survey of such kind of results, readers can refer to \cite{Sz}. We begin with some relevant results where the point set $U$ is not necessarily a Cartesian product. The following theorem was proved by R\'edei \cite{LR} in the case $|U|=p$, and later extended by Sz\H{o}nyi \cite{Sz1,Sz} to any $|U| \leq p$.

\begin{thm}[Theorem 5.2 in \cite{Sz}]\label{oldthm}
Let $p$ be a prime, and let $U\subset AG(2,p)$ with $1<|U|\leq p$. Then either $U$ is contained in a line, or $U$ determines at least $\frac{|U|+3}{2}$ directions.
\end{thm}

When the underlying field becomes $\F_q$, the problem becomes much more difficult. Sz\H{o}nyi \cite{Sz} proved the following interesting result, which is of a similar flavor as Theorem \ref{oldthm}. However, when we are working in $\F_q$, there are cases where $|D|$ is small.

\begin{thm}[Theorem 4 in \cite{Sz1}] \label{t0}
Let $U \subset AG(2,q)$ with $|U|=q-k$, where $0 \leq k \leq \sqrt{q}/2$. Then either $U$ determines at least $(q+1)/2$ directions, or it can be extended to a set $V$ with $|V|=q$, which determines the same set of directions as $U$.
\end{thm}

Note that in Theorem \ref{t0}, $|U|$ is assumed to be close to $q$. In general, $|U|$ could be much smaller compared to $q$, and the best-known result is the following theorem.

\begin{thm}[Theorem 1.3 in \cite{DD}]\label{dd}
Let $q=p^{s}$ be a prime power, and let $U \subset AG(2,q)$ with $1<|U|\leq q$. Then either $U$ is contained in a line or $U$ determines at least 
$\frac{|U|}{\sqrt{q}}$ directions if $s$ is even, and $\frac{|U|}{p^{\frac{s-1}{2}}+1}$ directions if $s$ is odd.
\end{thm}

We point out that R\'edei polynomial with Sz\H{o}nyi's extension is the main tool to prove the above theorems. Another key idea is to study the properties of lacunary polynomials, which are polynomials where there exists a substantial gap between the degree of two consecutive terms. In Section 2.1, we will describe these tools.

\subsection{Directions determined by a Cartesian product}

When the point set $U$ is a Cartesian product $A \times B$, we expect that the lower bound on $|D|$ can be improved, as $U$ is more structured. Let $A, B \subset \F_q$ be such that $|A|=m, |B|=n$. Denote $A=\{a_1,a_2, \ldots, a_m\}, B=\{b_1,b_2, \ldots, b_n\}$. The set of directions determined by $A \times B \subset AG(2, q)$ is 
\[ D = \frac{B-B}{A-A}=\left\{ \frac{y_2-y_1}{x_2-x_1} \colon x_1,x_2 \in A, y_1,y_2 \in B \right \} \subset \F_q \cup \{\infty\},\]
where for a set $X$, we denote $X-X=\{x_1-x_2: x_1,x_2 \in X\}$. Estimating the size of the set $D$ determined by certain Cartesian products (in particular $A \times A$) turns out to be useful in sum-product estimates over finite fields; see \cite{MPRS,RS,SS} for more details and examples. In this paper, we focus on improving the lower bound on $|D|$.  

Note that if $m=1$ or $n=1$, the direction set $D$ is trivial. And if $mn>q$, a simple pigeonhole argument  shows that $D=\F_q \cup \{\infty\}$. %If $mn$ is exactly $q$, then we can drop one element from $A$ or $B$. 
Also note that the set of directions only depends on the set $A-A, B-B$. Without loss of generality, we always assume that $m,n \geq 2$, $k=q-mn>0$, and $b_n=0$.

When $U=A \times B$, it turned out Theorem \ref{oldthm} can be significantly improved. In \cite{BSW}, Di Benedetto,  Solymosi, and White showed the following theorem. 
\begin{thm}[Theorem 1 of \cite{BSW}] \label{t1}
Let $A,B\subset \F_p$ be sets each of size at least two such that $|A||B| < p$. Then the set of points $A\times B\subset AG(2,p)$ determines at least $|A||B| - \min\{|A|,|B|\} + 2$ directions.
\end{thm}
Observe that the key lemma used in their proof is the following lemma.

\begin{lem}[Lemma 6 of \cite{BSW}] \label{e}
Let $R,S \in \F_p[x]$ be polynomials each with constant term 1.  Suppose that $R$ and $R'$ are relatively prime and $R$ does not divide $S$. If $x^{\text{deg}(R)+\text{deg}(S)+1}$ divides $R^m(x)S(x)-1$ for some $m$ not divisible by $p$, then $R(x) = 1$. 
\end{lem}

In general, it is possible that $R$ divides $S$. To use this lemma, we need to first write $S=R^rT$, where $r$ is the largest integer such that $R^r \mid S$. Then $T$ does not divide $S$, and $R^m(x)S(x)-1=R^{m+r}(x)T(x)-1$, so we can apply the lemma with $R$ and $T$. If we are working in $AG(2,p)$ and we wish to apply this lemma to estimate $|D|$, then we could expect $m+r<p$ and conclude that $R(x)=1$. Unfortunately, we fail to give effective bounds on $m+r$ when we are working in $AG(2,q)$. 

To extend their method to $AG(2,q)$, we need to generalize Lemma \ref{e}. In Section 3, we first prove Lemma \ref{l2} and then apply that to prove theorem \ref{t2}. The symmetric polynomials $f_{m,t}(b_1,b_2, \ldots, b_{n-1},0)$ in the statement of Theorem \ref{t2} will be defined via recurrence relations in Section 2.1, and we will give an explicit formula for $f_{m,t}$ in Section 2.2. Theorem \ref{t2} is central in proving our main results, Theorem \ref{t3} and Theorem \ref{t12}. 

\begin{thm} \label{t2}
Let $q=p^s$ be a prime power. Let $m,n \geq 2$ be integers such that $k=q-mn>0$. Let $A, B \subset \F_q$ with $|A|=m$ and $|B|=n$, and write $B=\{b_1,b_2, \ldots, b_{n-1},0\}$. Suppose $l$ is the smallest non-negative integer such that $f_{m,k-l}(b_1,b_2, \ldots, b_{n-1},0) \neq 0.$ Suppose one of the following conditions is satisfied:
\begin{enumerate}
    \item Every integer between $m$ and $m+\lfloor \frac{k-l}{n-1} \rfloor$ is not a multiple of $p$.
    \item $p \nmid (m+l)$. 
\end{enumerate}
Then the number of directions determined by the set $A \times B \subset AG(2,q)$ is at least $mn-n+l+2$.
\end{thm}
In Corollary \ref{thm1}, which is a corollary of Theorem \ref{t2}, it will be made precise that Theorem \ref{t2} is indeed a generalization of Theorem \ref{t1}. To apply Theorem \ref{t2}, it is important to understand the polynomial $f_{m,k}(r_1,r_2, \ldots, r_{n-1},0)$, especially the distribution of roots of $f_{m,k}$,  which we will discuss in Section 4.  In view of Corollary \ref{form}, which gives the explicit formula for $f_{m,k}$, we also need to study how binomial coefficients behave modulo the prime $p$. A useful tool in determining so is Lucas's Theorem. It states that 
if $p$ is a prime and if $m,n$ are non-negative integers with base-$p$ representation
$
m=m_{r}p^{r}+m_{r-1}p^{r-1}+\cdots +m_{1}p+m_{0}=(m_r, m_{r-1}, \ldots, m_{0})_p,
$
$
n=n_{r}p^{r}+n_{r-1}p^{r-1}+\cdots +n_{1}p+n_{0}=(n_r, n_{r-1}, \ldots, n_{0})_p,
$
where $0 \leq m_j,n_j \leq p-1$ for each $0 \leq j \leq r$, then 
$
\binom{m}{n} \equiv \prod_{j=0}^{r} \binom{m_j}{n_j} \pmod p.
$
Therefore, $\binom{m}{n} \not \equiv 0 \pmod p$ if and only if there is no carrying between the addition of $n$ and $m-n$ in base-$p$ representation. For an example of the application of Lucas's Theorem in estimating the number of directions determined by a point set in $AG(2,p^2)$, we refer to \cite{GLS}.

Furthermore, if we are working on $\F_q$, there must be some restriction on the sets $A,B$ so that we can conclude something similar to Theorem \ref{t1}. This is because if $E$ is a proper subfield of $\F_q$, and $A-A,B-B \subset E$, then all the directions determined by $A \times B \subset AG(2,q)$ are in $E \cup \{\infty\}$, and thus $|D| \leq |E|+1$.
Then the inequality $|A||B| - \min\{|A|,|B|\} + 2 \leq |D| \leq |E|+1$ fails to hold when $|A|, |B| \geq \sqrt{|E|}+1$. 

We will show that for given $A$ and $|B|$, it is very likely that the number of directions determined by $A \times B$ is close to $|A||B|$. The precise statement is given in the following theorem, which is our first main result, to be proved in Section 4.
\begin{thm} \label{t3}
Let $p \geq 3$ and $q=p^s$ be a prime power. Suppose $m \geq n \geq p$ and $k=q-mn>0$. Then for any $A\subset \F_q$ with $|A|=m$, if we choose an $n$-element set $B$ from $\F_q$ uniformly at random, we have 
$$
\Pr \bigg[\# \{ \text{directions in } A \times B \} \geq \frac{p-2}{p-1}(m-1)n+2\bigg] \geq 1-\frac{(q+(p-2)k-n)(q-1)^{n-2}}{(p-1)(q-1) \cdots (q-n+1)}.
$$
\end{thm}
Note that when $k$ is small compared to $q$, the lower bound of the above probability behaves like $\frac{p-2}{p-1}$. Compared to Theorem  \ref{dd}, we see that the lower bound on $|D|$ can be improved greatly when the point set is a Cartesian product.  Recall that Theorem \ref{t0} states that for a point set $U \subset AG(2,q)$, such that $|U|$ is close to $q$ (i.e. $k=q-|U|$ is small), there are two possibilities. The first one is the desired scenario, where we can conclude that the point set $U$ determines many directions. However, Theorem \ref{t0} does not predict how likely the desired scenario will happen. Theorem \ref{t3} gives us further insights in the conclusion of Theorem \ref{t0}. It implies that the desired scenario is very likely to occur, provided the point set $U$ is a Cartesian product $A \times B$.

\subsection{Clique number of Paley graphs and generalized Paley graphs}

For an undirected graph $G$, the {\em clique number} of $G$, denoted $\omega (G)$, is the size of a maximum clique of $G$.  Finding a reasonably good upper bound of the clique number of a Paley graph remains to be an open problem in additive combinatorics \cite{CL}. In the second half of the paper, we will discuss how to get improved upper bounds on the clique number of generalized Paley graphs over $\F_q$.  

We first define the (standard) Paley graph. Suppose $p$ a prime, such that $q=p^s \equiv 1 \pmod{4}$. The {\em Paley graph} on $\F_q$, denoted $P_q$, is the undirected graph whose vertices are elements in $\F_q$, such that two vertices are adjacent if and only if the difference of the two vertices is a square in $\F_q$. The trivial upper bound for $\omega(P_q)$ is $\sqrt{q}$. And when $q$ is a square, the trivial upper bound is tight \cite{BDR}.

For the case $q=p$, the current best result is the clique number of $P_p$ is at most $\sqrt{\frac{p}{2}}+1$, which was proved by Hanson and Petridis \cite{HP} using Stepanov's method. 
For the case that $q$ is an odd power of $p$, it is harder to improve the trivial upper bound. In \cite{BMR}, Bachoc, Ruzsa, and Matolcsi showed that $\omega(P_q) \leq \sqrt{q}-1$ for about non-square $q$. In \cite{Yip}, Yip extended the idea from Hanson and Petridis and improved the upper bound on $\omega(P_q)$ to $\min \bigg(p^s \bigg\lceil \sqrt{\frac{p}{2}} \bigg\rceil, \sqrt{\frac{q}{2}}+\frac{p^r+1}{4}+\frac{\sqrt{2p}}{32}p^{r-1}\bigg)$ for $q=p^{2r+1}$. For other relevant results on the clique number and other properties on the Paley graphs, we refer to the introduction section of \cite{Yip} and the survey paper \cite{ANE}. 

Similarly one can define generalized Paley graphs. They were first introduced by Cohen \cite{SC} in 1988, and reintroduced by Lim and Praeger \cite{LP} in 2009. Let $d>1$ be a positive integer. The {\em $d$-Paley graph} on $\F_q$, denoted $GP(q,d)$, is the undirected graph whose vertices are elements in $F_q$, where two vertices are adjacent if and only if the difference of the two vertices is a $d$-th power of $x$ for some $x \in \F_q$. Note that $2$-Paley graphs are just the standard Paley graphs. $3$-Paley graphs are also called {\em cubic Paley graphs} \cite{WA}.

One significant difference between Paley graphs and generalized Paley graphs is that when $d \geq 3$, $d$-Paley graphs lose some nice graph-theoretical properties that Paley graphs have (see \cite[Section 3.3]{ANE}). For example, Paley graphs are self-complementary and connected, while when $d\geq 3$, $d$-Paley graphs are not necessarily self-complementary or connected. This potentially makes it much more difficult to estimate the clique number of generalized Paley graphs. 

Similar to Paley graphs, the trivial upper bound for $\omega\big(GP(q,d)\big)$ is also $\sqrt{q}$; see Lemma \ref{tt}. Since there are only a few results on the estimates of the clique number of generalized Paley graphs, we will list all of them, and give some new bounds in Section 5. 
In particular, for certain $d$-Paley graphs over $\F_q$, we show that the clique number can be improved to $\sqrt{\frac{q}{d}}(1+o(1))$; see Theorem \ref{t6} for the precise statement.

Our second main result is an improved upper bound on the clique number of the cubic Paley graph over $\F_q$. We show that $\omega\big(GP(q,3)\big)$ can be improved to $0.769\sqrt{q}+1$, unless the clique number is $\sqrt{q}$ for obvious reasons (in which case the subfield $\F_{\sqrt{q}}$ is a maximum clique).

\begin{thm}\label{t13}
Let $q \equiv 1 \pmod 6$. If $q$ is not a square, then $\omega\big(GP(q,3)\big) < 0.718 \sqrt{q}+1$. If $q$ is a square, then $\omega\big(GP(q,3)\big)=\sqrt{q}$ if $3 \mid (\sqrt{q}+1)$ and $\omega\big(GP(q,3)\big)<0.769 \sqrt{q}+1$ otherwise.
\end{thm}

\subsection{Connection between the two problems}
The connection between the clique number of generalized Paley graphs of prime order and the number of directions determined by a Cartesian product in $AG(2, p)$ was first studied in \cite{BSW}. In fact, it is straightforward to use Theorem \ref{t1} to recover the Hanson-Petridis bound (Theorem \ref{HPBSW}) by the following observation: if $C$ is a clique of $GP(p,d)$, then the direction set determined by $C \times C \subset AG(2,p)$ is
$$D=\frac{C-C}{C-C} \subset (\F_p^*)^d \cup \{0,\infty\}.$$
This implies that $|D| \leq \frac{p-1}{\gcd(d,p-1)}+2$; combining this with the lower bound on $|D|$ given in Theorem \ref{t1}, we can establish an upper bound on $|C|$. 

It is clear that the same observation also works for $GP(q,d)$. Since we have obtained a similar result on $AG(2,q)$, we can also apply Theorem \ref{t2} to get an upper bound for generalized Paley graphs of prime power order. Unfortunately, for standard Paley graphs, the upper bound obtained in this way is much worse than the bound described in \cite{Yip}. In Section 6, we will establish a slightly complicated idea, which leads to improved bounds on $\omega\big(GP(q,d)\big)$. 

Let $\PP$ be the set of primes. For positive integers $r$ and $d$, we define $\QQ_{r,d}=\{p \in \PP: p^{2r+1} \equiv 1 \pmod {2d}\}$. In Section 6, utilizing an equidistribution result from analytic number theory, we obtain our third main result in this paper.

\begin{thm}\label{t12}
Let $h$ be a positive function such that $h(x)=o(x)$ as $x \to \infty$. Let $r,d$ be positive integers such that $d \geq 3$. Then $\omega\big(GP(p^{2r+1},d)\big) \leq p^{r+1/2}-h(p)$ for almost all $p \in \QQ_{r,d}$.
\end{thm}

\section{R\'edei polynomials with Sz\H{o}nyi's extension}

We mentioned that R\'edei polynomials are the main tools to estimate the size of the direction set in the introduction section. We begin by defining R\'edei polynomials.

\subsection{R\'edei polynomials}

The {\em R\'edei polynomial} of $A \times B \subset AG(2,q)$ is defined as 
$$ H(x,y) = \prod_{i=1}^m \prod_{j=1}^n (x+a_iy-b_j) .$$
For each $y \in \F_q$, define 
$A_y:=A_y(B)=\{-a_iy+b_j:1 \leq i \leq m, 1 \leq j \leq n\},$
as a multiset. Note that $x^q-x=\prod_{z \in \F_q} (x-z)$, so $H(x,y)$ divides $x^q-x$ if and only if the elements of $A_y$ are all distinct, which is equivalent to $y \not\in D$. We can write
$$
H(x,y)=\sum_{t=0}^{mn} (-1)^{mn-t}\sigma_{mn-t}(A_y) x^t
=x^{mn}-\sigma_1(A_y)x^{mn+1}+\cdots +(-1)^{mn}\sigma_{mn}(A_y),
$$
where $\sigma_j(A_y)$, $j=1,2, \cdots, mn$, are elementary symmetric polynomials on the multiset $A_y$.
When $y \not \in D$, Sz\H{o}nyi (see for example \cite{Sz}) extended R\'edei polynomial by introducing the polynomial $F(x,y)=(x^q-x)/H(x,y)$, where
\begin{equation} \label{eq1}
F(x,y) = x^k - \sigma_1(\F_q \setminus A_y)x^{k-1} + \sigma_2(\F_q \setminus A_y)x^{k-2} + \cdots +(-1)^m \sigma_k(\F_q \setminus A_y).
\end{equation}
Note that for each $0 \leq t \leq k$, $\sigma_t(A_y)$ is well-defined for a multiset $A_y$. However, it is not clear what is the meaning of $\sigma_t(\F_q \setminus A_y)$ for a multiset $A_y$. Next we follow the same idea in \cite{Sz} to show that it can be defined using a recurrence relation.

Observe that, when $y \not \in D$, for each $1 \leq t \leq k$, we have
$$
\sum_{j=0}^{t}\sigma_j(A_y)\sigma_{t-j}(\F_q \setminus A_y)=0.
$$
Therefore, for $y \not \in D$, we have the following recurrence relation for $\sigma_{t}(\F_q \setminus A_y)$:
$$
\sigma_0(\F_q \setminus A_y)=1,
$$
$$
\sigma_t(\F_q \setminus A_y)=-\sum_{j=1}^{t}\sigma_j(A_y)\sigma_{t-j}(\F_q \setminus A_y), \quad 1 \leq t \leq k.
$$
In this way, we see that $\sigma_t(\F_q \setminus A_y)$ is a polynomial in $y$ with degree at most $t$, and can be extended to be defined on all $y \in \F_q$. In this way, we can also extend $F(x,y)$ to be defined on all $y \in \F_q$ via the equation \eqref{eq1}. Let
\begin{equation}\label{eq2}
 H(x,y)F(x,y) = x^q + h_1(y)x^{q-1} + h_2(y)x^{q-2} + \cdots + h_q(y),
\end{equation}
and let $c_i=h_i(0)$ for each $1 \leq i \leq q$.
then $\operatorname{deg} (h_i) \leq i$. Next, we shall see how $H(x,y)$ and $F(x,y)$ can be used to obtain a lower bound on $|D|$. The proof of the following lemma is contained in Section 2 and Section 3 of \cite{BSW}. Here we include the proof for the sake of completeness.
\begin{lem}\label{l1}
If $c_i \neq 0$ for some $1 \leq i \leq q$, then  $|D| \geq q+1-i$.
\end{lem}
\begin{proof}
By the definition of the symmetric polynomials $\sigma_t(A_y)$ and $\sigma_t(\F_q \setminus A_y)$, we have $\operatorname{deg} (h_i) \leq i$. By definition, when $y \notin D$, $H(x,y)F(x,y) = x^q-x$, so we have $h_i(y) = 0$ for all $y \not\in D$. Since there are $q+1$ directions in $AG(2,q)$, and $\infty \in D$, there are $q+1-|D|$ directions not in $D$, and all such directions are in $\F_q$. This implies that $h_i \equiv 0$ for all $i< q+1-|D|$. Equivalently, if $h_i \not\equiv 0$ for some $1 \leq i \leq q$, then $|D| \geq q+1-i$. We proceed by setting $y=0$ in equation \eqref{eq2}:
\begin{equation} \label{key}
 H(x,0)F(x,0) = F(x,0)\prod_{j=1}^n (x-b_j)^m =x^q + c_1x^{q-1} + c_2x^{q-2} + \cdots + c_q.
\end{equation}
So if $c_i \neq 0$ for some $1 \leq i \leq q$, then $h_i \not \equiv 0$ and $|D| \geq q+1-i$.
\end{proof}

In \cite{BSW}, Lemma \ref{e} and Lemma \ref{l1} are combined to prove Theorem \ref{t1}. As we pointed out in the introduction section, Lemma \ref{e} is not strong enough for the application in $AG(2,q)$.
\subsection{Explicit formulas}
For our purpose, we would like to find an explicit formula for the symmetric polynomial $\sigma_t(\F_q \setminus A_y)$. Recall that $A_0=A_0(B)$ is the multiset $\{b_j: 1 \leq i \leq m, 1 \leq j \leq n\}=\cup_{j=1}^{n} \{b_j,b_j, \ldots, b_j\}$, where each $b_j$ appears $m$ times.
Next we revisit the the recurrence relation defined above. For example, when $t=1,2$, we have 
\begin{align*}
\sigma_1\big(\F_q \setminus A_0(B)\big)&=-\sigma_1\big(A_0(B)\big)=-m \sum_{j=1}^n b_j
=\binom{-m}{1} \sum_{j=1}^n b_j,    \\
\sigma_2\big(\F_q \setminus A_0(B)\big)
&=-\sigma_2\big(A_0(B)\big)-\sigma_1\big(A_0(B)\big)\sigma_1\big(\F_q \setminus A_0(B)\big)\\
&=-\sum_{1 \leq i<j \leq n} m^2b_i b_j -\binom{m}{2} \sum_{j=1}^n b_j^2 +m^2 (\sum_{j=1}^n b_j)^2\\
&=m^2 \sum_{1 \leq i<j \leq n} b_i b_j + \frac{m(m+1)}{2}\sum_{j=1}^n b_j^2\\
&=\binom{-m}{1} \binom{-m}{1} \sum_{1 \leq i<j \leq n} b_i b_j +\binom{-m}{2} \sum_{j=1}^n b_j^2.
\end{align*}

A pattern on the binomial coefficient could be conjectured based on the above computation, and we verify that in the following two lemmas.

\begin{lem} \label{binom}
If $1 \leq r \leq n$, $b_1=b_2=\cdots=b_r=1$ and $b_{r+1}=b_{r+2}=\cdots=b_n=0$, then for each $1 \leq t <q$,
$\sigma_t(\F_q\setminus A_0)=\binom{-mr}{t}$.
\end{lem}

\begin{proof}
We prove the statement by induction on $t$. For $t=1$, 
$$
\sigma_1(\F_q\setminus A_0)=-\sigma_1(A_0)=-\binom{mr}{1}=\binom{-mr}{1}.
$$
Suppose the statement is true for $t<l$, where $l \geq 2$, then by the recurrence relation, we have
\begin{align*}
\sigma_l(\F_q \setminus A_0)
&=-\sigma_l(A_0)-\sum_{j=1}^{l-1}\sigma_j(A_0)\sigma_{l-j}(\F_q \setminus A_0)\\
&=-\binom{mr}{l}-\sum_{j=1}^{l-1} \binom{mr}{j} \binom{-mr}{l-j}\\
&=-\sum_{j=1}^{l} \binom{mr}{j} \binom{-mr}{l-j}.    
\end{align*}
By Chu–Vandermonde identity for binomial coefficients,  
$$
\sum_{j=0}^{l} \binom{mr}{j} \binom{-mr}{l-j}=\binom{mr+(-mr)}{l}=0,
$$
so it follows that
\begin{align*}
\sigma_l(\F_q \setminus A_0)
&=0-\sum_{j=1}^{l} \binom{mr}{j} \binom{-mr}{l-j}\\
&=\sum_{j=0}^{l} \binom{mr}{j} \binom{-mr}{l-j}
-\sum_{j=1}^{l} \binom{mr}{j} \binom{-mr}{l-j}\\
&=\binom{-mr}{l}.    \qedhere
\end{align*}

\end{proof}

\begin{lem}
$\sigma_t\big(\F_q \setminus A_0(B)\big)$ is a homogeneous symmetric polynomial in $b_j$'s with degree $t$.
\end{lem}

\begin{proof}
From the definition of $\sigma_t\big(A_0(B)\big)$, it is either the zero polynomial or a homogeneous symmetric polynomial in $b_j$'s, with degree $t$. Then from the recurrence relation, inductively it is easy to show $\sigma_t\big(\F_q \setminus A_0(B)\big)$ is either the zero polynomial, or a homogeneous symmetric polynomial in $b_j$'s with degree $t$. And by Lemma \ref{binom},  if $b_1=b_2=\cdots=b_n=1$, then by Lucas's Theorem,
$$
\sigma_t\big(\F_q \setminus A_0(B)\big)=\binom{-mn}{t}=(-1)^t \binom{mn+t-1}{t}=(-1)^t \binom{q-1}{t} \neq 0.
$$
So $\sigma_t\big(\F_q \setminus A_0(B)\big)$ is not the zero polynomial, and the statement follows.
\end{proof}

Define $$f_{m,t}(b_1,b_2, \ldots, b_n)=\sigma_t\big(\F_q \setminus A_0(B)\big).$$
Note that $f_{m,t}$ does not depend on $A$, and $f_{m,t}$ is a homogeneous symmetric polynomial with degree $t$. Recall that for our purpose, we assume $b_n=0$. We would like to study the distribution of roots of $f_{m,t}(r_1,r_2, \ldots, r_{n-1},0)$, so we first need to check if this is a zero polynomial or not. If $f_{m,t}(r_1,r_2, \ldots, r_{n-1},0)$ is the zero polynomial, then all terms in $f_{m,t}$ without $r_n$ have zero coefficients. And since $f_{m,t}$ is symmetric, this implies that all terms in $f_{m,t}$ have zero coefficients except those terms with factors $r_1r_2 \cdots r_n$. In particular, this implies the following corollary. 
\begin{cor}\label{cor1}
If $f_{m,t}(r_1,r_2, \ldots, r_{n-1},0)$ is the zero polynomial, then 
$t \geq n$.
\begin{comment}
and 
$$
f_{m,t}(r_1,r_2, \cdots, r_{n-1},r_n)=r_1r_2 \cdots r_n g(r_1,r_2, \ldots, r_{n-1},r_n),
$$
where $g$ is a homogeneous symmetric polynomial with degree $t-n$. 
\end{comment}
\end{cor}
We will give an efficient algorithm to check whether $f_{m,t}(r_1,r_2, \ldots, r_{n-1},0)$ is the zero polynomial in the beginning of Section 4.

Now we are ready to find an explicit formula for $\sigma_t\big(\F_q \setminus A_0(B)\big)$, or $f_{m,t}(b_1,b_2, \ldots, b_{n-1},b_n)$. 

\begin{thm} \label{formula}
For each $1 \leq t <q$, 
$$
\sigma_{t}\big(\F_q \setminus A_0(B)\big)=\Sum_{\substack{r_1+r_2+\cdots+r_n=t \\r_i \geq 0}} \prod_{i=1}^{n} \binom{-m}{r_i} b_i^{r_i}.
$$
\end{thm}

\begin{proof}
We prove the statement by induction on $t$.
For each $t \geq 0$, by the definition of $\sigma_{t}(A_0)$, we have
$$
\sigma_t\big(A_0(B)\big)=\Sum_{\substack{\sum_{i=1}^{n} l_i=t \\l_i \geq 0}} \prod_{i=1}^{n} \binom{m}{l_i} b_i^{l_i}.
$$
And for $t=1$, the statement is true since
$$\sigma_{1}\big(\F_q \setminus A_0(B)\big)=-m(\sum_{i=1}^{n} b_i)= \sum_{i=1}^{n} \binom{-m}{1} b_i.
$$
Suppose the statement is true for $t<t_0$, where $t_0 \geq 2$, then for $t=t_0$, by the recurrence relation and inductive hypothesis, we have
\begin{align*}
&\Sum_{\substack{\sum_{i=1}^{n} r_i=t \\r_i \geq 0}} \prod_{i=1}^{n} \binom{-m}{r_i} b_i^{r_i}-\sigma_t\big(\F_q \setminus A_0(B)\big)\\
&=\Sum_{\substack{\sum_{i=1}^{n} r_i=t \\r_i \geq 0}} \prod_{i=1}^{n} \binom{-m}{r_i} b_i^{r_i}+\sigma_t\big(A_0(B)\big)+\sum_{j=1}^{t-1}\sigma_j\big(A_0(B)\big)\sigma_{t-j}\big(\F_q \setminus A_0(B)\big)\\
&=\Sum_{j=0}^{t} \Sum_{\substack{\sum_{i=1}^{n} l_i=j \\l_i \geq 0}} \prod_{i=1}^{n} \binom{m}{l_i} b_i^{l_i}    \Sum_{\substack{\sum_{i=1}^{n} r_i=t-j \\r_i \geq 0}} \prod_{i=1}^{n} \binom{-m}{r_i} b_i^{r_i}\\
%&=\Sum_{j=0}^{t} \Sum_{\substack{\sum_{i} l_i=j \\ \sum_{i} r_i=t-j \\l_i, r_i \geq 0}} \prod_{i=1}^{n} \binom{m}{l_i} \binom{-m}{r_i} b_i^{l_i+r_i} \\
&=\Sum_{\substack{\sum_{i=1}^{n} (l_i+r_i)=t \\l_i,r_i \geq 0}} \prod_{i=1}^{n} \binom{m}{l_i} \binom{-m}{r_i} b_i^{l_i+r_i}\\
&= \Sum_{\substack{\sum_{i=1}^{n} t_i=t \\t_i \geq 0}} \Sum_{\substack{0 \leq l_i \leq t_i}} \prod_{i=1}^{n} \binom{m}{l_i} \binom{-m}{t_i-l_i} b_i^{t_i}\\
&= \Sum_{\substack{\sum_{i=1}^{n} t_i=t \\t_i \geq 0}}  \prod_{i=1}^{n} b_i^{t_i} \bigg(\sum_{l_i=0}^{t_i} \binom{m}{l_i} \binom{-m}{t_i-l_i}\bigg).
\end{align*}
By Chu–Vandermonde identity, for each $1 \leq i \leq n$ and each $t_i \geq 0$,
$$
\sum_{l_i=0}^{t_i} \binom{m}{l_i} \binom{-m}{t_i-l_i}=\binom{m+(-m)}{t_i}=\binom{0}{t_i}
= \begin{cases}
0 &\text{$t_i>0$}\\
1 &\text{$t_i=0$}
\end{cases}.
$$
If $t_i \geq 0$ for each $1 \leq i \leq n$, and $\sum_{i=1}^{n} t_i=t \geq 2$, then there exists $i_0$ such that $t_{i_0} \geq 1$, so we have
$\prod_{i=1}^{n} \binom{0}{t_i}=0.$ It follows that
\[
\Sum_{\substack{\sum_{i=1}^{n} r_i=t \\r_i \geq 0}} \prod_{i=1}^{n} \binom{-m}{r_i} b_i^{r_i}-\sigma_t\big(\F_q \setminus A_0(B)\big)
= \Sum_{\substack{\sum_{i=1}^{n} t_i=t \\t_i \geq 0}}  \prod_{i=1}^{n} \binom{0}{t_i}b_i^{t_i} =0. \qedhere
\]
\end{proof}

\begin{cor} \label{form}
For each $1 \leq t<q$, 
$$
f_{m,t}(b_1,b_2, \ldots, b_{n-1},0)=(-1)^{t}\Sum_{\substack{t_1+t_2+\cdots+t_{n-1}=t \\t_i \geq 0}} \prod_{i=1}^{n-1} \binom{m+t_i-1}{m-1} b_i^{t_i}.
$$
\end{cor}

\begin{proof}
This follows immediately from Theorem \ref{formula}
and 
$$
\binom{-m}{t_i}=(-1)^{t_i} \binom{m+t_i-1}{t_i} =(-1)^{t_i} \binom{m+t_i-1}{m-1}.
$$
\end{proof}

\section{Directions determined by a Cartesian product in $AG(2,q)$}
In this section, we will prove Theorem \ref{t2}, and give some corollaries. We begin by giving a stronger version of Lemma \ref{e}.

\begin{lem}\label{l2} 
Let $q=p^s$ to be a prime power. Let $R,S \in \F_q[x]$ be non-constant polynomials each with constant term 1. Suppose that $R$ and $R'$ are relatively prime,  $m,n \geq 2$, $k=q-mn> 0$, $\operatorname{deg} R=n-1$, and $\operatorname{deg} S=k-l$ for some integer $0 \leq l \leq k$. 
If one of the following conditions is satisfied:
\begin{enumerate}
    \item Every integer between $m$ and $m+\lfloor \frac{k-l}{n-1} \rfloor$ is not a multiple of $p$. 
    \item $p \nmid (m+l)$.
\end{enumerate}
Then  $x^{\operatorname{deg}R+\operatorname{deg}S+1}$ does not divide $R^m(x)S(x)-1$.
\end{lem}

\begin{proof} We use proof by contradiction. Suppose there exists a polynomial $P(x) \in \F_q[x]$ such that
\begin{equation}
R^m(x) S(x) = 1 + x^{\operatorname{deg}R+\operatorname{deg}S+1}P(x).
\end{equation}
Let $r$ be the highest power of $R$ dividing $S$. Then $0 \leq r \leq \lfloor \frac{k-l}{n-1} \rfloor$. Let $T=\frac{S}{R^r}$, then $R$ does not divide $T$, and we have
\begin{equation}\label{eq} 
R^{m+r}(x) T(x) = 1 + x^{\operatorname{deg}R+\operatorname{deg}S+1}P(x)=1+ x^{n+k-l} P(x).
\end{equation}
By differentiating (\ref{eq}), we obtain
\[ R^{m+r-1}(x) \big( (m+r)R'(x)T(x) + R(x) T'(x) \big) = x^{n+k-l-1} \big((n+k-l)P(x) + xP'(x)\big).\]
Since the constant term in $R^{m+r-1}(x)$ is 1, we see that $x^{n+k-l-1}$ divides $(m+r)R'(x)T(x) + R(x) T'(x) $. But the degree of $(m+r)R'(x)T(x) + R(x) T'(x)$ is at most $n+k-l-2$, so we must have
$$
(m+r)R'(x)T(x) + R(x) T'(x) =(n+k-l)P(x) + xP'(x) =0.
$$
Since $R$ and $R'$ are relatively prime, then $R(x) \mid (m+r)T(x)$. And since $R$ does not divide $T$, we must have $m+r=0$ in $\F_q$, i.e. $p \mid (m+r)$. Note that $m \leq m+r \leq m+\lfloor \frac{k-l}{n-1} \rfloor$, so there is a integer between $m$ and $m+\lfloor \frac{k-l}{n-1} \rfloor$ which is a multiple of $p$. Moreover, we must also have $R(x)T'(x)=0$. Since $\F_q[x]$ is an integral domain, and $R(x)$ has constant term 1, then it follows that $T'(x)=0$. Therefore $T(x)=g(x^p)$ for some polynomial $g \in \F_q[x]$, and in particular, 
$$p \mid \operatorname{deg} T= \operatorname{deg} S-r \operatorname{deg} R=k-l-r(n-1)=q-mn-l-r(n-1),$$
combining with $p \mid (m+r)$, we obtain that $p \mid (m+l)$. 
\end{proof}

We remark that we actually proved a slightly stronger statement: if $p \nmid (m+r)$ or $p \nmid (m+l)$, where $r$ is the highest power of $R$ dividing $S$, then $x^{\operatorname{deg}R+\operatorname{deg}S+1}$ does not divide $R^m(x)S(x)-1$. However, the exact value or $r$ is difficult to compute without knowing the explicit factorizations of polynomials $R$ and $S$, which is indeed the case in our application.

Lemma \ref{l1}, Lemma \ref{l2} can be combined to prove Theorem \ref{t2}.

\begin{comment}
\noindent\textbf{Theorem~\ref{t2}.}
\begin{em}
Let $q=p^s$ be a prime power. Let $m,n \geq 2$ be integers such that $k=q-mn>0$. Let $A, B \subset \F_q$ with $|A|=m$ and $|B|=n$, and write $B=\{b_1,b_2, \ldots, b_{n-1},0\}$. Suppose $l$ is the smallest non-negative integer such that $f_{m,k-l}(b_1,b_2, \ldots, b_{n-1},0) \neq 0.$ Suppose one of the following conditions is satisfied:
\begin{enumerate}
    \item Every integer between $m$ and $m+\lfloor \frac{k-l}{n-1} \rfloor$ is not a multiple of $p$.
    \item $p \nmid (m+l)$. 
\end{enumerate}
Then the number of directions determined by the set $A \times B \subset AG(2,q)$ is at least $mn-n+l+2$.
\end{em}
\end{comment}

\begin{proof}[Proof of Theorem~\ref{t2}]
We will consider equation \eqref{eq1} and \eqref{key}. Suppose that $c_1 = c_2 = \cdots = c_{k+n-l-1} = 0$. Set $R(y) = \prod_{j=1}^{n-1} (1-b_jy)$, and $S(y) = y^kF(y^{-1},0)$. Then $R(y),S(y) \in \F_q[y]$, and $\operatorname{deg}R = n-1$. Note that $f_{m,0}(b_1,b_2, \ldots, b_{n-1},0)=1$, and since $l$ is the smallest non-negative integer such that $f_{m,k-l}(b_1,b_2, \ldots, b_{n-1},0) \neq 0$, then $l \leq k$, and $\operatorname{deg} S=k-l$.  Substitute $x = y^{-1}$ in and multiply by $y^q$ in \eqref{key} to obtain 
\begin{equation}\label{eq3}
R^m(y)S(y) = 1 + c_1y + c_2y^2 + \cdots + c_qy^q = 1 + y^{k+n-l} U(y),    
\end{equation}
for some polynomial $U(y) \in \F_q[y]$. Since the elements of $B$ are distinct, all roots of $R$ have multiplicity 1, and $R$ is relatively prime to $R'$. However, given one of the conditions in the statement, equation \eqref{eq3} is impossible to hold in view of Lemma \ref{l2}. It follows that at least one of $c_1,\ldots,c_{k+n-l-1}$ is nonzero, and thus by Lemma \ref{l1}, there are at least $q-(k+n-l-1)+1 = mn-n+l+2$ directions determined by $A \times B$. 
\end{proof}

In particular, when $q=p$, we get a slightly stronger version of Theorem \ref{t1}.

\begin{cor}\label{thm1}
Let $p$ be a prime. Let $m \geq n \geq 2$ be integers such that $k=p-mn>0$. Let $A, B \subset \F_p$ with $|A|=m$ and $|B|=n$, and write $B=\{b_1,b_2, \ldots, b_{n-1},0\}$. Suppose $l$ is the smallest non-negative integer such that $f_{m,k-l}(b_1,b_2, \ldots, b_{n-1},0) \neq 0$, then the number of directions determined by the set $A \times B \subset AG(2,p)$ is at least $mn-n+l+2$.
\end{cor}

\begin{proof}
Note that $l \leq k$, so $0<m+l \leq m+k <2m+k \leq mn+k=p$. This implies that $p \nmid (m+l)$. So by Theorem \ref{t2}, the number of directions determined by the set $A \times B$ is at least $mn-n+l+2$.
\end{proof}

The following are some special cases where we can conclude the same lower bound on the number of directions without any additional assumptions.

\begin{cor}
Let $p$ be a prime. Let $m,n$ be integers such that $2 \leq m<p<n$ and $k=p^2-mn>0$. Let $A, B \subset \F_{p^2}$ with $|A|=m$ and $|B|=n$. Then the number of directions determined by the set $A \times B \subset AG(2,p^2)$ is at least $mn-n+2$.
\end{cor}
\begin{proof}
We have
$$m+\bigg\lfloor \frac{k}{n-1} \bigg\rfloor \leq m+\bigg\lfloor\frac{p^2-mn}{n-1}\bigg\rfloor=\bigg\lfloor\frac{p^2-m}{n-1}\bigg\rfloor \leq \bigg\lfloor\frac{p^2-2}{p}\bigg\rfloor<p.
$$ So by Theorem \ref{t2}, the number of directions is at least $mn-n+2$.
\end{proof}

\begin{cor}\label{corr}
Let $q=p^s$ be a prime power. Let $A, B \subset \F_q$ with $|A|=m, |B|=n$, where $m,n \geq 2$ are integers such that $p \nmid m$ and $0<k=q-mn <n-1$. Then the number of directions determined by the set $A \times B \subset AG(2,q)$ is at least $mn-n+2$. In particular, if $p \nmid m$, $2 \leq m \leq \sqrt{q}-1$, and $n=\lfloor\frac{q}{m} \rfloor$, then the number of directions determined by the set $A \times B \subset AG(2,q)$ is at least $mn-n+2$.
\end{cor}

\begin{proof}
Suppose $l$ is the smallest non-negative integer such that $f_{m,k-l}(b_1,b_2, \ldots, b_{n-1},0) \neq 0$.
Since $\lfloor \frac{k-l}{n-1} \rfloor \leq \lfloor \frac{k}{n-1} \rfloor=0$, and $p \nmid m$, then the condition (1) in Theorem \ref{t2} is satisfied, so the number of directions is at least $mn-n+2$. In particular, if $p \nmid m$, and $m \geq 2$, then $m \nmid q$, and thus $0<q-mn=k<m$. Since $m \leq \lfloor \sqrt{q} \rfloor -1 $, then $n \geq \lfloor \sqrt{q} \rfloor +1 \geq m+2$. Thus $k<n-1$, and the conclusion follows.
\end{proof}

\section{Number of roots of $f_{m,k}(r_1,r_2, \ldots, r_{n-1},0)$}

To apply Theorem \ref{t2}, it is crucial to understand when is  $f_{m,k}(b_1,b_2, \ldots, b_{n-1},0)=0$. In particular, one need to identify whether $f_{m,k}(r_1,r_2, \ldots, r_{n-1},0) \equiv 0$. Recall Corollary \ref{cor1} says that $f_{m,k}(r_1,r_2, \ldots, r_{n-1},0) \equiv 0$ could happen only when $k \geq n$. 
\subsection{Polynomial identity testing}
In general, we can use Schwartz–Zippel Lemma as a tool to design a randomized algorithm to test whether a given multivariate polynomial is the zero polynomial (see for example \cite{JS}). However, since we have worked out the explicit formula in Corollary \ref{form}, we have the following deterministic and efficient algorithm, Algorithm \ref{alg}, to check whether $f_{m,k}(r_1,r_2, \ldots, r_{n-1},0) \equiv 0$. We need the following simple lemma as a preparation.

\begin{lem}\label{carry}
Let $d \geq 2$ be a fixed positive integer. Suppose $k \geq 0$ and $a_0,a_1, \ldots, a_k \geq 0$ such that $A:=a_0+a_1d+a_2d+\ldots+a_kd^k>d^{k+1}$, then there exist $b_0,b_1, \ldots, b_k$ such that $0 \leq b_j \leq a_j$ for each $0 \leq j \leq k$, and $b_0+b_1d+b_2d+\ldots+b_kd^k=A-d^{k+1}$.
\end{lem}

\begin{proof}
We prove by inducting on $k$. The case $k=0$ is trivial. Suppose $k \geq 1$ and $a_0,a_1, \ldots, a_k \geq 0$ such that $A:=a_0+a_1d+a_2d+\ldots+a_kd^k>d^{k+1}$. If $a_k \geq d$, then we can set $b_j=a_j$ for $0 \leq j \leq k-1$ and $b_k=a_k-d$ so that $b_0+b_1d+b_2d+\ldots+b_kd^k=A-d^{k+1}$.
Next assume $a_k<d$, and let $l=d-a_k$, $b_k=0$. Let $B=a_0+a_1d+a_2d+\ldots+a_{k-1}d^{k-1}$, then $B>l d^{k}$. By inductive hypothesis, there exists $b_0,b_1, \ldots, b_{k-1}$ such that $0 \leq b_j \leq a_j$ for each $0 \leq j \leq k-1$, and $b_0+b_1d+b_2d+\ldots+b_{k-1}d^{k-1}=B-ld^{k}$. Then it follows that $b_0+b_1d+b_2d+\ldots+b_{k-1}d^{k-1}+b_d^{k}=B-ld^{k}=A-(a_k+l)d^k=A-d^{k+1}$.
\end{proof}

\begin{prop}\label{check}
Suppose $1 \leq t<q$. Let $m-1=(m_{s-1}, m_{s-2}, \ldots, m_0)_p$, and $t=(h_{s-1},h_{s-2}, \ldots, h_0)_p$ be the base-$p$ representation of $m-1$ and $t$, respectively.
The following algorithm can detect whether $f_{m,t}(r_1,r_2, \ldots, r_{n-1},0) \equiv 0$. Moreover, the running time is $O(\log q)$.

\begin{algorithm}[H] \label{alg}
$S_0 \gets 0$\\
\For{$j \gets 0$ to $s-1$}
{
$S_j \gets S_j+(n-1)(p-1-m_j)$\\
\eIf{$S_j<h_j$}{
    \Return ``zero polynomial''}
  {
    $S_{j+1} \gets \lfloor \frac{S_j-h_j}{p} \rfloor$
  }
}
\Return ``nonzero polynomial''
\caption{Check whether $f_{m,t}(r_1,r_2, \ldots, r_{n-1},0)$ is the zero polynomial.}
\end{algorithm}
\end{prop}

\begin{proof}
It is clear that the running time of the above algorithm is $O(s)=O(\log q)$. By Corollary \ref{form}, $f_{m,t}(r_1,r_2, \ldots, r_{n-1},0)$ is not the zero polynomial if and only if there exist $t_1,t_2, \ldots,t_{n-1} \geq 0$ such that
\begin{equation}\label{eq4}
\sum_{i=1}^{n-1} t_i=t,\prod_{i=1}^{n-1} \binom{m-1+t_i}{m-1} \not\equiv 0 \pmod p.
\end{equation}

Note that $t<q=p^s$. Fix $T_0, T_1, \ldots, T_{s-1} \geq 0$. Let $t_1,t_2, \ldots,t_{n-1}$ be such that $0 \leq t_i<q$, with base-$p$ representations
$t_i=(g_{s-1,i}, g_{s-2,i}, \ldots, g_{0,i})_p$ for each $1 \leq i \leq n-1$ satisfying $T_j=\sum_{i=0}^{n-1} g_{j,i}$ for each $0 \leq j \leq s-1$. By Lucas's Theorem, $\binom{m-1+t_i}{m-1} \not\equiv 0 \pmod p$ if and only if there is no carrying in the addition of $m-1$ and $t_i$ in the base-$p$ representation. Therefore, $\prod_{i=1}^{n-1} \binom{m-1+t_i}{m-1} \not\equiv 0 \pmod p$ if and only if $g_{j,i}$ takes value between $0$ and $p-1-m_j$ for each $1 \leq i \leq n-1$ and $0 \leq j \leq s-1$. It follows that there exist $t_1,t_2, \ldots,t_{n-1}$ such that $0 \leq t_i<q$ and $\prod_{i=1}^{n-1} \binom{m-1+t_i}{m-1} \not\equiv 0 \pmod p$ if and only if $T_j \leq (n-1)(p-1-m_j)$ for each $0 \leq j \leq s-1$. 

Let 
$$
T_0+T_1p+ T_2p^2+\ldots+ T_{s-1}p^{s-1}=R_0+R_1p+R_2p^2+\ldots+R_{s-1}p^{s-1}+R_sp^s,
$$
where $0 \leq R_j<p$ for each $0 \leq j \leq s-1$, and $R_s \geq 0$. Note that $$T_0+T_1p+ T_2p^2+\ldots+ T_{s-1}p^{s-1}=\sum_{i=1}^{n-1} t_i \equiv t \pmod q$$ is equivalent to $$\sum_{i=1}^{n-1} t_i \equiv t \pmod p, \sum_{i=1}^{n-1} t_i \equiv t \pmod {p^2}, \ldots, \sum_{i=1}^{n-1} t_i \equiv t \pmod {p^s}.
$$ 
Therefore, there exist $t_1,t_2, \ldots,t_{n-1}$ such that $0 \leq t_i<q$ and $\sum_{i=1}^{n-1} t_i \equiv t \pmod q$ if and only if $R_j=h_j$ for each $0 \leq j \leq s-1$. 

It is clear that for each $0 \leq j \leq s-1$, the $S_j$ computed in the above algorithm is exactly the maximum value of $R_j$ provided $T_k \leq (n-1)(p-1-m_k)$ for each $0 \leq k \leq j$ and $\sum_{i=1}^{n-1} t_i \equiv t \pmod {p^j}$, where $\lfloor \frac{S_j-h_j}{p} \rfloor$ is the maximum number of carries between the addition of $t_1,t_2, \ldots,t_{n-1}$ from $p^j$ digit to the $p^{j+1}$ digit. In particular, if \eqref{eq4} holds for $t_1,t_2, \ldots, t_{n-1}$, then $T_j \leq (n-1)(p-1-m_j)$, and $S_j \geq R_j=h_j$ for each $0 \leq j \leq s-1$. Therefore, if $S_j<h_j$ for some $0 \leq j \leq s-1$, then $f_{m,t}(r_1,r_2, \ldots, r_{n-1},0)$ is the zero polynomial, and Algorithm \ref{alg} correctly returns ``zero polynomial''.

Conversely, suppose Algorithm \ref{alg} returns ``nonzero polynomial'', then $S_j \geq h_j$ for each $0 \leq j \leq s-1$. Furthermore, there are $T_0, T_1, \ldots, T_{s-1}$ (which are maximized) such that $0 \leq T_j \leq (n-1)(p-1-m_j)$ for each $j$ and $$T_0+T_1p+ T_2p^2+\ldots+ T_{s-1}p^{s-1}=h_0+h_1p+h_2p^2+\ldots+h_{s-1}p^{s-1}+S_sp^s=t+S_sp^s$$ for $S_s=\lfloor \frac{S_{s-1}-h_{s-1}}{p} \rfloor \geq 0$ given in the Algorithm \ref{alg}. Since $S_s \geq 0$, by Lemma \ref{carry}, there exist $T'_0,T'_1, \ldots, T'_{s-1}$ such that $0 \leq T'_j \leq T_j$, and 
$$
T'_0+T'_1p+ T'_2p^2+\ldots+ T'_{s-1}p^{s-1}=h_0+h_1p+h_2p^2+\ldots+h_{s-1}p^{s-1}=t.
$$
It follows that there exist $t_1,t_2, \ldots,t_{n-1}$ such that $0 \leq t_i<q$ and \eqref{eq4} holds. Therefore, $f_{m,t}(r_1,r_2, \ldots, r_{n-1},0)$ is a nonzero polynomial, and Algorithm \ref{alg} returns the correct answer. 
\end{proof}

Below we see a family of pairs $(m,t)$ where Algorithm \ref{alg}  returns ``nonzero polynomial''.

\begin{cor}
Let $m-1=(m_{s-1}, m_{s-2}, \ldots, m_0)_p$, and $t=(h_{s-1},h_{s-2}, \ldots, h_0)_p$ be the base-$p$ representation of $m-1$ and $t$, respectively. If $m_j \neq p-1$ for each $0 \leq j \leq s-1$, and $n-1 \geq \max\{h_j:0 \leq j \leq s-1\}$, then $f_{m,t}(r_1,r_2, \ldots, r_{n-1},0)$ is a nonzero polynomial.
\end{cor}

\begin{proof}
For each $0 \leq j \leq s-1$, since $m_j \neq p-1$, we have
$S_j \geq (n-1)(p-1-m_j) \geq n-1 \geq h_j$. Then by Proposition \ref{check},  $f_{m,t}(r_1,r_2, \ldots, r_{n-1},0)$ is a nonzero polynomial.
\end{proof}

In particular, when $n \geq p$, we have $n-1 \geq \max\{h_j:0 \leq j \leq s-1\}$. Thus, we obtain the following corollary.

\begin{cor}\label{cor5}
If $n \geq p$ and the base-$p$ representation of $m-1$ does not contain $p-1$, then $f_{m,t}(r_1,r_2, \ldots, r_{n-1},0)$ is a nonzero polynomial.
\end{cor}

The conditions in the above corollary might not hold for all $m$, but $m$ can be always reduced slightly to make that feasible.

\begin{lem}\label{red}
If $p \geq 3$, then for any $2 \leq m<q$, there is $m'<m$ such that $(m'-1) \geq \frac{p-2}{p-1} (m-1)$, $p \nmid m'$ and the base-$p$ representation of $m'-1$ does not contain $p-1$.
\end{lem}

\begin{proof}
Let $m-1=(m_{s-1}, m_{s-2}, \ldots, m_0)_p$. Let $j_0$ be the largest integer such that $m_{j_0}=p-1$. Let $m'=1+(m_{s-1}, \ldots, m_{j_0+1},p-2, \ldots, p-2)_p$. Then the base-$p$ representation of $m'-1$ does not contain $p-1$, $p \nmid m'$, and $m-1 \leq (m_{s-1}, \ldots, m_{j_0+1},p-1, \ldots, p-1)_p$. So we have
$$
\frac{m'-1}{m-1} \geq \frac{(m_{s-1}, \ldots, m_{j_0+1},p-2, \ldots, p-2)_p}{(m_{s-1}, \ldots, m_{j_0+1},p-1, \ldots, p-1)_p} \geq \frac{p-2}{p-1}.
$$
\end{proof}

We will use a combination of Corollary \ref{cor5} and Lemma \ref{red} to prove Theorem \ref{ttt}. 

\subsection{Upper bounds on the number of roots}
We aim to find a lower bound for the probability on $f_{m,t}(r_1,r_2, \ldots, r_{n-1},0) \neq 0$. The following lemma is useful in bounding the number of roots of a nonzero multivariate polynomial.

\begin{lem}[Schwartz–Zippel Lemma, Corollary 1 in \cite{JS}]
Let $g\in F[x_{1},x_{2},\ldots ,x_{n}]$ be a non-zero polynomial with degree $d$ over a field F. Let $S$ be a finite subset of $F$ and let $r_{1},r_{2},\ldots ,r_{n}$ be selected at random independently and uniformly from $S$. Then
$$
 \Pr [g(r_{1},r_{2},\ldots ,r_{n})=0] \leq {\frac {d}{|S|}}.
$$
\end{lem}

Next, we use Schwartz–Zippel Lemma to bound the number of roots with distinct coordinates.

\begin{prop} \label{prob}
Let $1 \leq t < q$. Suppose $f_{m,t}(r_{1},r_{2},\ldots ,r_{n-1},0)$ is not the zero polynomial, if we choose a $(n-1)$-set $B'=\{b_{1},b_{2},\ldots ,b_{n-1}\}$ from $\F_q^*$ uniformly at random, then 
$$
 \Pr [f_{m,t}(b_{1},b_{2},\ldots ,b_{n-1},0)=0] \leq \frac{t(q-1)^{n-2}}{(q-1) \cdots (q-n+1)}.
$$
\end{prop}

\begin{proof}
Since $f_{m,t}(r_{1},r_{2},\ldots ,r_{n-1},0)$ is not the zero polynomial, it is a symmetric polynomial with degree $t$.
Let $S=\F_q^*$, then by Schwartz–Zippel Lemma, if we pick $r_1,r_2, \ldots, r_{n-1}$ from $S$ independently and uniformly, we have
$$
\Pr [f_{m,t}(r_{1},r_{2},\ldots ,r_{n-1},0)=0] \leq {\frac {t}{q-1}}.
$$
So the number of $(n-1)$-tuples $(r_{1},r_{2},\ldots,r_{n-1}) \in S^{n-1}$ such that $f_{m,t}(r_{1},r_{2},\ldots ,r_{n-1},0)=0$ is at most $\frac{t}{q-1} (q-1)^{n-1}=t(q-1)^{n-2}$. If $f_{m,t}(b_{1},b_{2},\ldots ,b_{n-1},0)=0$, then since $f_{m,t}$ is a symmetric polynomial, we also have $f_{m,t}(b_{\pi(1)},b_{\pi(2)},\ldots ,b_{\pi(n-1)},0)=0$ for any permutation $\pi \in \operatorname{Sym}(n-1)$. So the number of $(n-1)$-sets $B'=\{b_1,b_2, \ldots, b_{n-1}\}$ of $\F_q^*$ such that $f_{m,t}(b_{1},b_{2},\ldots ,b_{n-1},0)=0$ is at most $\frac{t(q-1)^{n-2}}{(n-1)!}$. Since the number of $(n-1)$-sets $B'$ of $\F_q^*$ is $\binom{q-1}{n-1}$, if we choose a $(n-1)$-set $B'$ from $\F_q^*$ uniformly at random, then 
$$
{\displaystyle \Pr [f_{m,t}(b_{1},b_{2},\ldots ,b_{n-1},0)=0] \leq \frac{t(q-1)^{n-2}}{(n-1)! \binom{q-1}{n-1}}}=\frac{t(q-1)^{n-2}}{(q-1) \cdots (q-n+1)}.
$$
\end{proof}
\subsection{Proof of Theorem \ref{t3}} In this subsection, we will prove Theorem \ref{t3}. There are two different cases: $f_{m,k}(r_{1},r_{2},\ldots ,r_{n-1},0) \equiv 0$ and $f_{m,k}(r_{1},r_{2},\ldots ,r_{n-1},0) \not \equiv 0$.

If $f_{m,k}(r_{1},r_{2},\ldots ,r_{n-1},0)$ is not the zero polynomial (which is the case when $k<n$, by Corollary \ref{cor1}), then by combining Theorem \ref{t2} and Proposition \ref{prob}, we have the following estimate on the probability.
\begin{thm} \label{res}
Let $q=p^s$ be a prime power. Let $m,n \geq 2$ be integers such that $k=q-mn>0$. Suppose $p \nmid m$ and $f_{m,k}(r_{1},r_{2},\ldots ,r_{n-1},0)$ is not the zero polynomial (in particular when $k<n$; in general, this can be checked efficiently by Algorithm \ref{alg} in $O(\log q)$ time). Then for any $A\subset \F_q$ with $|A|=m$, if we choose a $n$-set $B$ from $\F_q$ uniformly at random, we have 
$$
\Pr [\# \{ \text{directions in } A \times B \} \geq mn-n+2] \geq 1-\frac{k(q-1)^{n-2}}{(q-1) \cdots (q-n+1)}.
$$
\end{thm}

If $f_{m,k}(r_1,r_2, \ldots, r_{n-1},0)$ is indeed the zero polynomial, then in view of Theorem \ref{t2}, we need to find the smallest positive integer $l$ such that $f_{m,k-l}(r_1,r_2, \ldots, r_{n-1},0)$ is a nonzero polynomial. Recall that Corollary \ref{cor1} states that $f_{m,t}(r_1,r_2, \ldots, r_{n-1},0) \equiv 0$ could happen only when $t \geq n$, so such $l$ exists. We can run Algorithm \ref{alg} to check that for each $l$ using brute force, which takes at most $O(ks)=O(k \log q)$ time. In this way, by using Theorem \ref{t2} and Proposition \ref{prob} with $t=k-l$, we obtain the following theorem.

\begin{thm} \label{res2}
Let $q=p^s$ be a prime power. Let $m,n \geq 2$ be integers such that $k=q-mn>0$. Suppose $l$ is the smallest non-negative integer such that $f_{m,k-l} (r_1,r_2, \ldots, r_{n-1},0)$ is not the zero polynomial. If $p \nmid (m+l)$, then for any $A\subset \F_q$ with $|A|=m$, if we choose a $n$-set $B$ from $\F_q$ uniformly at random, we have 
$$
\Pr [\# \{ \text{directions in } A \times B \} \geq mn-n+2] \geq 1-\frac{(k-l)(q-1)^{n-2}}{(q-1) \cdots (q-n+1)}.
$$
\end{thm}

However, it is still possible that $p \mid (m+l)$. In which case our approach is to reduce the parameter $m$ slightly to obtain a nonzero polynomial by the observation in Corollary \ref{cor5} and Lemma \ref{red}. Note that reducing $m$ corresponds to discarding some elements from $A$, which only decreases the number of directions determined.

\begin{thm}\label{ttt}
Let $p \geq 3$ and $q=p^s$ be a prime power. Let $m,n$ be integers such that $m \geq n \geq p$ and $k=q-mn>0$. Suppose $l$ is the smallest non-negative integer such that $f_{m,k-l} (r_1,r_2, \ldots, r_{n-1},0)$ is not the zero polynomial. If $p \mid (m+l)$, then for any $A\subset \F_q$ with $|A|=m$, if we choose a $n$-set $B$ from $\F_q$ uniformly at random, we have 
$$
\Pr \bigg[\# \{ \text{directions in } A \times B \} \geq \frac{p-2}{p-1}(m-1)n+2\bigg] \geq 1-\frac{(q+(p-2)k-n)(q-1)^{n-2}}{(p-1)(q-1) \cdots (q-n+1)}.
$$
\end{thm}
\begin{proof}
By Lemma \ref{red}, there is $m'<m$ such that $(m'-1) \geq \frac{p-2}{p-1} (m-1)$, $p \nmid m'$ and the base-$p$ representation of $m'-1$ does not contain $p-1$. Then $m' \geq 1+\frac{p-2}{p-1}>1$, so $m' \geq 2$. Let $A'$ be any subset of $A$ with $|A'|=m'$, then by Corollary \ref{cor5}, the polynomial $f_{m',k'}(r_1,r_2, \ldots, r_{n-1},0)$ associated to the set $A'$ and $k'=q-m'n$, is a nonzero polynomial. Note that $$k'=q-m'n \leq q-\frac{p-2}{p-1}(m-1)n-n=q-\frac{p-2}{p-1}\bigg(\frac{q-k}{n}-1\bigg)n-n=\frac{q+(p-2)k-n}{p-1}.$$
Since $p \nmid m'$, and $A' \subset A$, by Theorem \ref{res},  we have
\begin{align*}
&\Pr \bigg[\# \{ \text{directions in } A \times B \} \geq \frac{p-2}{p-1}(m-1)n+2\bigg] \\
&\geq \Pr [\# \{ \text{directions in } A' \times B \} \geq (m'-1)n+2] \\
&\geq 1-\frac{k'(q-1)^{n-2}}{(q-1) \cdots (q-n+1)}.\\
&\geq 1-\frac{(q+(p-2)k-n)(q-1)^{n-2}}{(p-1)(q-1) \cdots (q-n+1)}. \qedhere
\end{align*}
\end{proof}

In particular, if we do not bother the exact value of $l$, then we can combine Theorem \ref{res2} and Theorem \ref{ttt} to get a slightly weaker version, which is Theorem~\ref{t3}.

\begin{comment}
\noindent\textbf{Theorem~\ref{t3}.}
\begin{em}
Let $p \geq 3$ and $q=p^s$ be a prime power. Suppose $m \geq n \geq p$ and $k=q-mn>0$. Then for any $A\subset \F_q$ with $|A|=m$, if we choose an $n$-element set $B$ from $\F_q$ uniformly at random, we have 
$$
\Pr \bigg[\# \{ \text{directions in } A \times B \} \geq \frac{p-2}{p-1}(m-1)n+2\bigg] \geq 1-\frac{(q+(p-2)k-n)(q-1)^{n-2}}{(p-1)(q-1) \cdots (q-n+1)}.
$$
\end{em}
\end{comment}

\section{Clique Number of Generalized Paley Graphs}
%We shall first introduce the definition of generalized Paley graphs, and the notations defined here will be used in all of the following discussion.

Let $p$ be an odd prime and $s$ a positive integer such that $q=p^s$. Recall two vertices of $GP(q,d)$ are adjacent if and only if their difference is a $d$-th power. It is clear that if $\gcd(d,q-1)=\gcd(d',q-1)$, then $GP(q,d)$ and $GP(q,d')$ are isomorphic graphs since $\F_q^*$ is a cyclic group. So we can replace $d$ by $\gcd(d,q-1)$, and assume $d \mid (q-1)$. Also note that in order for $GP(q,d)$ to be a undirected graph, we need $-1$ to be a $d$-th power in $\F_q^*$, i.e. $\frac{q-1}{d}$ to be an even number. 

In the following discussion, we will always assume $d>1$ and $d$ is a divisor of $\frac{q-1}{2}$, or equivalently $q \equiv 1 \pmod {2d}$. Let $N=\omega\big(GP(q,d)\big)$ and let $C=\{v_1, v_2, \ldots, v_N\} \subset \F_q$ be a clique of the maximum size in $GP(q,d)$. We are interested in finding a reasonably good lower and upper bound for the clique number.
\subsection{Known bounds} We begin by giving some trivial upper bounds for the clique number in the case $d \geq 3$.

\begin{lem}\label{t}
If $q \equiv 1 \pmod {2d}$, then $\omega\big(GP(q,d)\big) \leq \frac{q-1}{d}+1$.
\end{lem}

\begin{proof}
Note that $v_2-v_1,v_3-v_1, \ldots, v_N-v_1$ are distinct nonzero $d$-th powers in $\F_q^*$ and the number of $d$-th powers in $\F_q^*$ is $\frac{q-1}{d}$. So $\omega\big(GP(q,d)\big) \leq \frac{q-1}{d}+1$.
\end{proof}

In the literature \cite{BMR,BDR,SC,CL,HP}, the trivial upper bound on $\omega\big(GP(q,d)\big)$ is given by $\sqrt{q}$. Here we include a short proof for completeness.

\begin{lem}\label{tt}
If $q \equiv 1 \pmod {2d}$, then $\omega\big(GP(q,d)\big) \leq \sqrt{q}$.
\end{lem}
\begin{proof}
Let $g$ be a be a primitive root of $\F_q^*$, and consider the set $W=\{v_i+gv_j: 1 \leq i,j \leq N\}$. 
Note that if $v_i+gv_j=v_i'+gv_j'$, then $v_i-v_i'=g(v_j'-v_j)$, which is impossible unless $i=i'$ and $j=j'$. So each element of $W$ is different from the others. This means that $|W|=N^2 \leq q$, i.e. $N \leq \sqrt{q}$.
\end{proof}

In \cite{SC}, Cohen proved the following theorem on the lower bound of clique number.

\begin{thm}[Theorem 3 in \cite{SC}]
If $d \geq 3$ and $q \equiv 1 \pmod {2d}$, then $\omega\big(GP(q,d)\big) \geq \frac{p}{(p-1)\log d}(\frac{1}{2} \log q-2\log \log q)-1$.
\end{thm}

The lower bound Cohen obtained is of the order $\log q$, which is significantly smaller compared to the trivial upper bound. The following theorem shows that the lower bound can be greatly improved in certain cases. %We include the proof for the sake of completeness.

\begin{thm}[Theorem 1 in \cite{BDR}]\label{t4}
Let $q \equiv 1 \pmod {2d}$, and let $r$ be the largest integer such that $d \mid \frac{q-1}{p^r-1}$, then $\omega\big(GP(q,d)\big) \geq p^r$.
\end{thm}
\begin{comment}
\begin{proof}
Note that $(p^r-1) \mid (q-1)$ implies $r \mid s$, so  $\F_{p^r}$ is a subfield of $\F_q$. We claim that $\F_{p^r}$ is a $d$-Paley clique in $GP(q,d)$. To prove that, it suffices to show that each $x \in \F_{p^r}$ is a $d$-th power. Let $g$ be a primitive root of $\F_q^*$, then $g$ has order $q-1$, and the nonzero elements of $\F_{p^r}$ are exactly the roots of $x^{p^r-1}=1$. Let $x=g^l \in \F_{p^r}^*$, then $x^{p^r-1}=1$, which is equivalent to $(q-1) \mid l(p^r-1)$, so $d \mid l$, i.e. $x$ is a $d$-th power.
\end{proof}
\end{comment}

Combining Theorem \ref{t4} and the trivial upper bound Lemma \ref{tt}, we get the following.

\begin{cor}\label{cor3}
When $q$ is a square and $d \mid (\sqrt{q}+1)$, $\omega\big(GP(q,d)\big)=\sqrt{q}$.
\end{cor}

This means that Lemma \ref{tt} gives the best trivial upper bound, in the sense that we cannot improve it without any additional assumption. Theorem \ref{t4} also implies that the lower bound $q^{1/d}$ can be obtained in the following cases.

\begin{prop} \label{lb}
If $\gcd(d, \phi(d))=1$, $2d \mid (q-1)$ and $d \mid s$, then $\omega\big(GP(q,d)\big) \geq q^{1/d}$. In particular, if $d$ is a prime such that $2d \mid (q-1)$ and $d \mid s$, then $\omega\big(GP(q,d)\big) \geq q^{1/d}$.
\end{prop}
\begin{proof}
Let $\delta$ be the order of $p$ modulo $d$. Then by Euler's  Theorem, we have $d \mid (p^{\phi(d)}-1)$, so $\delta \mid \phi(d)$ and $\gcd(\delta,d)=1$ since $\gcd(d, \phi(d))=1$. On the other hand, since $d\mid (q-1)$, we have $\delta \mid s$. Now $d \mid s$ and $\gcd(\delta,d)=1$ imply $\delta \mid \frac{s}{d}$, so $p^{s/d} \equiv 1 \pmod d$, and 
we have $$\frac{q-1}{p^{s/d}-1}=\frac{p^s-1}{p^{s/d}-1}=1+p^{s/d}+p^{2s/d}+\cdots+p^{(d-1)s/d} \equiv d \equiv 0 \pmod d.$$ So by Theorem \ref{t4}, we have $\omega\big(GP(q,d)\big) \geq p^{s/d}=q^{1/d}$.
\end{proof}

\subsection{Stepanov's method and binomial coefficients}

In \cite{HP}, Hanson and Petridis used Stepanov's method to improve the upper bound on $\omega\big(GP(p,d)\big)$. In \cite{BSW}, Di Benedetto, Solymosi, and White recovered the same bound. 
\begin{thm}[Corollary 1.5 in \cite{HP}, Corollary 2 in \cite{BSW}] \label{HPBSW}
Let $p$ be a prime such that $p \equiv 1 \pmod {2d}$, then
$\omega^2\big(GP(p,d)\big)-\omega\big(GP(p,d)\big) \leq \frac{p-1}{d}$. Equivalently, $\omega\big(GP(p,d)\big) \leq \sqrt{\frac{p-1}{d}+\frac{1}{4}}+\frac{1}{2}$.
\end{thm}
Note that both methods only work in the prime fields. In \cite{Yip}, Yip extended Hanson and Petridis' method to improve the trivial upper bound on the clique number of Paley graphs of prime power order, by carefully analyzing the binomial coefficients. Actually, in certain cases, a similar idea also leads to an improved upper bound for generalized Paley graphs. Similar to \cite[Theorem 1.6]{Yip}, we have the following theorem for generalized Paley graphs.

\begin{thm} \label{t5}
If $q\equiv 1 \pmod{2d}$, and $2 \leq n\leq N=\omega\big(GP(q,d)\big)$ satisfies
$
\binom{n-1+\frac{q-1}{d}}{\frac{q-1}{d}}\not \equiv 0 \pmod p,
$
then $(N-1)n \leq \frac{q-1}{d}$.
\end{thm}

\begin{proof}
 Consider the following polynomial 
$$
f(x)=\sum_{i=1}^n c_i (x-v_i)^{n-1+\frac{q-1}{d}} -1 \in \F_q[x],
$$
where $c_1,c_2,...,c_n$ is the unique solution of the following system of equations:
\[
\left\{
\TABbinary\tabbedCenterstack[l]{
\sum_{i=1}^n c_i (-v_i)^j=0,  \quad 0 \leq j \leq n-2\\\\
\sum_{i=1}^n c_i (-v_i)^{n-1}=1
}\right.
\]
Note the above system of equation has a unique solution since the coefficient matrix of the system is a Vandermonde matrix with parameters $v_1, v_2, \ldots v_n$ all distinct.  Similar to the proof of Theorem 1.6 in \cite{Yip}, we can show that the degree of $f$ is $\frac{q-1}{d}$, each of $v_1,v_2, \ldots v_n$ is a root of $f$ of multiplicity at least $n-1$, and each of $v_{n+1},v_{n+2}, \ldots v_N$ is a root of $f$ of multiplicity at least $n$. 
Therefore
\[
n(n-1)+(N-n)n= (N-1)n \leq \operatorname{deg}f =\frac{q-1}{d}. \qedhere
\]
\end{proof}

The following Corollary shows that Theorem \ref{t5} is a generalization of Theorem \ref{HPBSW}.

\begin{cor}
If $q\equiv 1 \pmod{2d}$, and $N=\omega\big(GP(q,d)\big)$ satisfies
$
\binom{N-1+\frac{q-1}{d}}{\frac{q-1}{d}}\not \equiv 0 \pmod p,
$
then $\omega\big(GP(q,d)\big) \leq \sqrt{\frac{q-1}{d}+\frac{1}{4}}+\frac{1}{2}$. In particular, if $p \equiv 1 \pmod {2d}$, then $\omega\big(GP(p,d)\big) \leq \sqrt{\frac{p-1}{d}+\frac{1}{4}}+\frac{1}{2}$.
\end{cor}
\begin{proof}
If $
\binom{N-1+\frac{q-1}{d}}{\frac{q-1}{d}}\not \equiv 0 \pmod p,
$
then we can take $n=N$ in Theorem \ref{t5} to conclude that $(N-1)N \leq \frac{q-1}{d}$, i.e. $N \leq \sqrt{\frac{q-1}{d}+\frac{1}{4}}+\frac{1}{2}$.
When $q$ is a prime, note that by Lemma \ref{t}, $N=\omega\big(GP(p,d)\big)\leq \frac{p-1}{d}+1$, then $N-1+\frac{p-1}{d}\leq \frac{2(p-1)}{d} \leq p-1<p$ and therefore $
\binom{N-1+\frac{p-1}{d}}{\frac{p-1}{d}}\not \equiv 0 \pmod p.
$
\end{proof}

\subsection{Improved bounds on the clique number of certain generalized Paley graphs}

In this subsection, we will extend the idea in \cite{Yip} to obtain improved bounds on  $\omega\big(GP(q,d)\big)$. In particular, we will prove Theorem \ref{t13}, which shows that for $\omega\big(GP(q,3)\big)$, the trivial bound $\sqrt{q}$ can be improved to $0.769\sqrt{q}+1$.

We need to deal with the case when $q$ is a prime power. We can assume $\sqrt{q}\geq N>\sqrt{\frac{q-1}{d}+\frac{1}{4}}+\frac{1}{2}$. In view of Theorem \ref{t5}, we need to determine the largest $n \leq N$ such that $
\binom{n-1+\frac{q-1}{d}}{\frac{q-1}{d}}\not \equiv 0 \pmod p.$ 
Again, our main tool is Lucas's Theorem. For each given $q$ and $d$, we shall have no difficulty finding the desired $n$ by hand. However, in general, the analysis will be much more complicated than the case $d=2$ (standard Paley graph). For example, it highly depends on the base-$p$ representation of $\frac{q-1}{d}$ and the size of $\log_q d$, as we need to compare the number of digits of the the base-$p$ representations of $\frac{q-1}{d}$, $\lfloor \sqrt{q} \rfloor$ and $\bigg\lceil \sqrt{\frac{q-1}{d}+\frac{1}{4}}+\frac{1}{2} \bigg \rceil $. 

We first focus on the case $d \mid (p-1)$. In this case, the base-$p$ representation of $\frac{q-1}{d}$ is simply 
$$\frac{q-1}{d}=\bigg(\frac{p-1}{d},\frac{p-1}{d}, \ldots, \frac{p-1}{d}\bigg)_p.$$
We need to deal with the cases $s$ is odd and $s$ is even separately because $\sqrt{q}$ behaves very differently in both cases. When $s$ is odd, we can mimic the proof of Theorem 3.5 in \cite{Yip}.

\begin{thm}\label{t6}
If $q=p^{2r+1} \equiv 1 \pmod {2d}$, $d \geq 3, r \geq 1$, and $d \mid (p-1)$, then $$\omega\big(GP(q,d)\big) < \sqrt{\frac{q}{d}} \bigg(1+\frac{(d-1)^2}{8dp} + \frac{1}{2} \big(1-\frac{1}{d}\big)\sqrt{\frac{d}{p}} \bigg)+1.$$
\end{thm}

\begin{proof}
Since $d \geq 3$, we have $p \geq 7$. In view of Lemma \ref{tt}, we can assume that $\sqrt{p}\cdot p^r \geq N >\sqrt{\frac{p}{d}}\cdot p^r$. Let the base-$p$ representation of $N-1$ be $
N-1=(z_r,z_{r-1},...,z_0)_p,
$
then $\sqrt{\frac{p}{d}} \leq z_r \leq \sqrt{p}$. Note that $z_r+\frac{p-1}{d} \leq \sqrt{p}+\frac{p-1}{d} \leq \sqrt{p}+\frac{p-1}{3}<p$ since $p \geq 7$. 
\begin{itemize}
\item If $z_{r-1}+\frac{p-1}{d}\leq p-1$, we can take  
$n-1=z_rp^r+z_{r-1}p^{r-1}$. Then $N-p^{r-1}+1 \leq n \leq N \leq n$ if $r \geq 2$, and $n=N$ if $r=1$. And Lucas's Theorem implies that
$$
\binom{n-1+\frac{q-1}{d}}{\frac{q-1}{d}} \equiv \binom{z_r+{\frac{p-1}{d}}}{{\frac{p-1}{d}}}\binom{z_{r-1}+\frac{p-1}{d}}{\frac{p-1}{d}} \not \equiv 0 \pmod p.
$$
\item If $z_{r-1}+\frac{p-1}{d}> p-1$, let 
$n-1=z_rp^r+p^r-1-\frac{p^r-1}{d}$, then  
$N-p^r+\frac{p^r-1}{d} \leq n \leq N \leq n$
and 
$$
\binom{n-1+\frac{q-1}{d}}{\frac{q-1}{d}} \equiv \binom{z_r+{\frac{p-1}{d}}}{{\frac{p-1}{d}}} 
\binom{p-1}{\frac{p-1}{d}}^{r}\not \equiv 0 \pmod p.
$$
\end{itemize}
To conclude, we can always find $N-p^r+\frac{p^r-1}{d} \leq n \leq N$ such that $\binom{n-1+\frac{q-1}{d}}{\frac{q-1}{d}} \not \equiv 0 \pmod p$. Then by Theorem \ref{t5}, we have
 $(N-1)(N-p^r+\frac{p^r-1}{d}) \leq (N-1)n \leq \frac{q-1}{d}$, so $N^2-(p^r+1-\frac{p^r-1}{d})N \leq \frac{q+p^r-2}{d}-p^r$ and therefore 
 \begin{align*}
 N&\leq \sqrt{\frac{q+p^r-2}{d}-p^r+\frac{1}{4}\bigg(p^r+1-\frac{p^r-1}{d}\bigg)^2}+\frac{1}{2}\bigg(p^r+1-\frac{p^r-1}{d}\bigg)\\
 &=\sqrt{\frac{q}{d}+\frac{1}{4}p^{2r}\bigg(1-\frac{1}{d}\bigg)^2-p^r\bigg(1-\frac{1}{d}+\frac{1}{2}-\frac{1}{2d^2}\bigg) 
+\frac{1}{4}\bigg(1+\frac{1}{d}\bigg)^2}+\frac{1}{2}\bigg(p^r+1-\frac{p^r-1}{d}\bigg)\\
&<\sqrt{\frac{q}{d}}+\bigg(1-\frac{1}{d}\bigg)^2\frac{\sqrt{d}}{8}p^{r-1/2} +\frac{1}{2} +\frac{1}{2}\bigg(p^r+1-\frac{p^r}{d}\bigg)\\
&= \sqrt{\frac{q}{d}} \bigg(1+\frac{(d-1)^2}{8dp} + \frac{1}{2} \big(1-\frac{1}{d}\big)\sqrt{\frac{d}{p}} \bigg)+1.
 \end{align*}
\end{proof}

In the case $q$ is a square, $d \mid (p-1)$ would imply $q \equiv 1 \pmod {2d}$, so we do not need to assume that explicitly. Recall that for the (standard) Paley graph over $\F_q$, the clique number attains the trivial upper bound $\sqrt{q}$ if $q$ is a square. Next, we show this is not the case for generalized Paley graphs. We will give a better bound in Theorem \ref{t7}.
\begin{lem}
If $q$ is a square, $d \geq 3$ and $d \mid (p-1)$, then $\omega\big(GP(q,d)\big) \leq \sqrt{q}-1$.
\end{lem}
\begin{proof}
Let $q=p^{2r}$. In view of Lemma \ref{tt}, it suffices to show that $N \neq p^r$. Suppose $N=p^r$, then we can take $n=p^r-\frac{p^r-1}{d}<N$ such that 
$$n-1+\frac{q-1}{d}= \bigg(\frac{p-1}{d}, \ldots, \frac{p-1}{d},p-1, \ldots, p-1\bigg)_p,$$
$$
\binom{n-1+\frac{q-1}{d}}{\frac{q-1}{d}} \equiv \binom{\frac{p-1}{d}}{\frac{p-1}{d}}^r \binom{p-1}{\frac{p-1}{d}}^r \not \equiv 0 \pmod p.
$$
Then by Theorem \ref{t5}, we have $(N-1)n \leq \frac{q-1}{d}$, i.e. $(p^r-1)(p^r-\frac{p^r-1}{d})\leq \frac{p^{2r}-1}{d}$. This implies $dp^r-(p^r-1) \leq p^r+1$, i.e. $d \leq 2$, a contradiction. 
\end{proof}

\begin{thm}\label{t7}
If $q$ is a square, $d \geq 3$ and $d \mid (p-1)$, then $$\omega\big(GP(q,d)\big) < \sqrt{\frac{q}{d}} \big(1+\frac{1}{2\sqrt{d}}+\frac{1}{8d}\big)+1.$$
\end{thm}

\begin{proof}
Let $q=p^{2r}$. We can assume that $p^r-1 \geq N >\sqrt{\frac{p^2}{d}}\cdot p^{r-1}$. Let the base-$p$ representation of $N-1$ be 
$
N-1=(z_{r-1},z_{r-2},...,z_0)_p,
$
then $\sqrt{\frac{p^2}{d}} \leq z_{r-1} \leq p-1$. 
\begin{itemize}
    \item If $z_{r-1}+\frac{p-1}{d}<p$, then we can take $n-1=z_{r-1}p^{r-1}$. We have $N-p^{r-1}+1 \leq n \leq N$ and 
$$
\binom{n-1+\frac{q-1}{d}}{\frac{q-1}{d}} \equiv \binom{z_{r-1}+{\frac{p-1}{d}}}{{\frac{p-1}{d}}} \not \equiv 0 \pmod p.
$$ 
    \item If $z_{r-1}+\frac{p-1}{d}\geq p$, then we can take $n-1=p^r-1-\frac{p^r-1}{d}$. We have $N-\frac{p^r-1}{d} \leq n \leq N$ and
$$
\binom{n-1+\frac{q-1}{d}}{\frac{q-1}{d}} \equiv \binom{p^r-1}{{\frac{p^r-1}{d}}} \equiv \binom{p-1}{\frac{p-1}{d}}^r \not \equiv 0 \pmod p.
$$
\end{itemize}
To conclude, we can always find $N-\frac{p^r-1}{d} \leq n \leq N$ such that $\binom{n-1+\frac{q-1}{d}}{\frac{q-1}{d}} \not \equiv 0 \pmod p$. Then by Theorem \ref{t5}, we have $(N-1)(N-\frac{p^r-1}{d}) \leq (N-1)n \leq \frac{q-1}{d}$, so $N^2-(\frac{p^r-1}{d}+1)N \leq \frac{q+p^r-2}{d}$ and therefore 
\begin{align*}
N
&\leq \sqrt{\frac{q+p^r-2}{d}+\frac{1}{4}\bigg(\frac{p^r-1}{d}+1\bigg)^2}+\frac{1}{2}\bigg(\frac{p^r-1}{d}+1\bigg)\\
&=\sqrt{\frac{q}{d}+\frac{p^{2r}}{4d^2}+p^r\bigg(\frac{1}{d}+\frac{d-1}{2d^2}\bigg)+\frac{(d-1)^2}{4d^2}-\frac{2}{d}}+\frac{1}{2}\bigg(\frac{p^r-1}{d}+1\bigg)    \\
&<\sqrt{\frac{q}{d}+\frac{p^{2r}}{4d^2}+\frac{3p^r}{2d}+\frac{1}{4}}+\frac{1}{2}\bigg(\frac{p^r-1}{d}+1\bigg) \\
&<\sqrt{\frac{q}{d}}+\frac{p^r}{8d\sqrt{d}}+\frac{1}{2}+\frac{1}{2}\bigg(\frac{p^r-1}{d}+1\bigg) \\
&<\sqrt{\frac{q}{d}} \bigg(1+\frac{1}{2\sqrt{d}}+\frac{1}{8d}\bigg)+1. \qedhere
\end{align*}
\end{proof}
%Note that when $d=2$, if we use this method, the factor of $\sqrt{q}$ will be greater than 1. This is consistent because the clique number of a standard Paley graph is $\sqrt{q}$ when $q$ is a square. 
Note that when $d \geq 3$,
$$\frac{1}{\sqrt{d}}+\frac{1}{2d}+\frac{1}{8d\sqrt{d}} \leq \frac{1}{\sqrt{3}}+\frac{1}{6}+\frac{1}{24\sqrt{3}}<0.769,$$
so this bound is always better than the trivial bound. 
\begin{comment}
Next we consider the case that $d \nmid (p-1), d \mid (p^2-1)$. In this case, $\frac{q-1}{d}$ will have period 2 in its base-$p$ representation
\begin{thm}
Let $q$ be a square and let $3 \leq d<p$ such that $d \nmid (p-1), d \mid (p^2-1)$ and $d \nmid (\sqrt{q}+1)$, then $\omega\big(GP(q,d)\big) < \sqrt{\frac{q}{d}} \big(1+\frac{1}{2\sqrt{d}}+\frac{1}{8d}\big)+1.$
\end{thm}
\end{comment}

In general, given $d \geq 3$, to estimate $\omega\big(GP(q,d)\big)$ using Theorem \ref{t5}, we need to determine all possible values of the order of $p$ modulo $d$. If the order is $\delta \mid \phi(d)$, then $\frac{q-1}{d}$ will be periodic in base-$p$ representation, with period $\delta$, and we can try to apply Theorem \ref{t5} to obtain an upper bound on the clique number. It should be clear that the analysis will be very complicated when the number of divisors of $\phi(d)$ is large. We demonstrate this process for cubic Paley graphs and prove Theorem~\ref{t13}.

\begin{comment}
\noindent\textbf{Theorem~\ref{t13}.}
\begin{em}
Let $q \equiv 1 \pmod 6$. If $q$ is not a square, then $\omega\big(GP(q,3)\big) < 0.718 \sqrt{q}+1$. If $q$ is a square, then $\omega\big(GP(q,3)\big)=\sqrt{q}$ if $3 \mid (\sqrt{q}+1)$ and $\omega\big(GP(q,3)\big)<0.769 \sqrt{q}+1$ otherwise.
\end{em}
\end{comment}

\begin{proof}[Proof of Theorem~\ref{t13}]
Let $q=p^s$. Since $q \equiv 1 \pmod 6$, then either $p \equiv 1 \pmod 3$, or $p \equiv 2 \pmod 3$ and $s$ is an even integer.

If $p \equiv 1 \pmod 3$, then $p \geq 7$. If $s$ is odd, then by Theorem \ref{t6}, $$\omega\big(GP(q,3)\big)<\sqrt{\frac{q}{d}} \bigg(1+\frac{1}{6p} + \frac{1}{3} \sqrt{\frac{3}{p}} \bigg)+1 <0.718 \sqrt{q}+1.$$ If $s$ is even, then by Theorem \ref{t7}, $\omega\big(GP(q,3)\big)<0.769 \sqrt{q}+1$.

If $p \equiv 2 \pmod 3$, and $s$ is even, then we can set $s=2r$. Let $N=\omega\big(GP(q,3)\big)$. If $r$ is odd, then $3 \mid (\sqrt{q}+1)$ and thus by Corollary \ref{cor3}, $N=\sqrt{q}$. Next we assume $r$ is even. We have
    $$\frac{q-1}{3}=\bigg(\frac{p-2}{3},\frac{2p-1}{3},\frac{p-2}{3},\frac{2p-1}{3}, \ldots, \frac{p-2}{3},\frac{2p-1}{3}\bigg)_p.$$ We can assume that $p^r \geq N >\sqrt{\frac{p^2}{d}}\cdot p^{r-1}$. Let the base-$p$ representation of $N-1$ be
$
N-1=(z_{r-1},z_{r-2},...,z_0)_p,
$
then $\sqrt{\frac{p^2}{d}} \leq z_{r-1} \leq p-1$. 
\begin{itemize}
    \item If $z_{r-1}+\frac{p-2}{3}<p$, then we can take $n-1=z_{r-1}p^{r-1}$. We have $N-p^{r-1}+1 \leq n \leq N$ and 
$$
\binom{n-1+\frac{q-1}{3}}{\frac{q-1}{3}} \equiv \binom{z_{r-1}+{\frac{p-2}{3}}}{{\frac{p-2}{3}}} \not \equiv 0 \pmod p.
$$ 
    \item If $z_{r-1}+\frac{p-2}{3}\geq p$, then we can take $n=p^r-\frac{p^r-1}{3}$. We have $N-\frac{p^r-1}{3} \leq n \leq N$ and
$$
\binom{n-1+\frac{q-1}{3}}{\frac{q-1}{3}} \equiv \binom{p^r-1}{{\frac{p^r-1}{3}}} \equiv \binom{p-1}{\frac{p-2}{3}}^{r/2} \binom{p-1}{\frac{2p-1}{3}}^{r/2} \not \equiv 0 \pmod p.
$$
To conclude, we can always find $n$ such that $N-\frac{p^r-1}{3} \leq n \leq N$ and $\binom{n-1+\frac{q-1}{3}}{\frac{q-1}{3}} \not \equiv 0 \pmod p$. Similar to the computation in the proof of Theorem \ref{t7}, we have $N<\sqrt{\frac{q}{3}} \big(1+\frac{1}{2\sqrt{3}}+\frac{1}{24}\big)+1<0.769 \sqrt{q}+1$. \qedhere
\end{itemize}
\end{proof}

%We remark that applying the above argument repeatedly does not improve the the $0.769\sqrt{q}+1$ bound, because it is still possible that $z_{r-1} +\frac{p-2}{3} \geq p$. 

Using Proposition \ref{lb}, we see that the clique number of certain cubic Paley graphs is at least $q^{1/3}$. For such cubic Paley graphs, it is an open question to improve the range $[q^{1/3}, 0.769\sqrt{q}+1]$ on the clique number.

\section{Proof of Theorem \ref{t12}}

In this section, we make use of Theorem \ref{t2} and an equidistribution result from analytic number theory to prove our third main result, Theorem \ref{t12}.
\subsection{Equidistribution results involving prime powers}
A sequence $\{y_n: n \in \N\} \subset \R$ is called equidistributed modulo 1 if for any $\alpha \in [0,1]$, we have $\lim_{n \to \infty} \frac{Z(n, \alpha)}{n}=\alpha$, where $Z(n, \alpha)=\#\{y_j: 1\leq j \leq n, \{y_j\} \leq \alpha\}$. Let $e(x)=\exp(2\pi i x)$. The characterization of equidistributed sequences is given by the following well-known Weyl's criterion.
\begin{lem}[Weyl's criterion]\label{Weyl}
A sequence $\{y_n\}$ is equidistributed if and only if for any integer $t \neq 0$, $\sum_{n \leq x}e(t y_n)=o(x)$ as $x \to \infty$.
\end{lem}

Similar to the $1$-dimensional case, we can also define the notion of equidistribution in a similar way for the multidimensional case, and we also have the multidimensional Weyl's criterion (see for example Section 1.6 of \cite{KN}).

Recall we denote  $\PP$ to be the set of primes (with the natural order). Let $g$ be a nice function, we would like to show the sequence $(g(p))_{p \in \PP}$ is equdistributed modulo 1. By Weyl's criterion and partial summation, it suffices to show that for any non-zero integer $t$, we have
\begin{equation}
\sum_{n \leq x}e(t g(n)) \Lambda(n) =o(x), \text{ as } x \to \infty.
\end{equation}
To estimate the exponential sum of the above form, it is standard to use van der Corput's method and Vaughan's identity (see for example chapter 8 and chapter 13 in \cite{IK}). In particular, When $g(x)=\sqrt{x}$, for any $\alpha \neq 0$, we have (see page 348 of \cite{IK})
$$
\sum_{n\le x} e(\alpha \sqrt{n}) \Lambda(n) \ll_\alpha x^{\frac{5}{6}} (\log x)^4.
$$
Therefore, $(\sqrt{p})_{p \in \PP}$ is equidistributed  modulo 1. In general, we have the following equidistribution result involving prime powers. 

\begin{thm}[Corollary 2.1 in \cite{BKMST}]\label{t11}
Let $\xi(x)=\sum_{j=1}^m \alpha_j x^{\theta_j}$, where $0 <\theta_1 <\theta_2< \cdots< \theta_m$, $\alpha_j$ are nonzero real numbers. Assume that if all $\theta_j\in \N$, then at least one $\alpha_j$ is irrational. Then for any $h \in \Z$, the sequence $(\xi(p-h))_{p \in \PP}$ is equdistributed modulo 1.
\end{thm}

For any two positive integers $a$ and $b$, we define $\PP_{a,b}=\PP \cap (a\Z+b)$. In \cite[Corollary 2.3]{BKMST}, a stronger version of Theorem \ref{t11} is proved. It basically states that the sequence is still equdistributed when we restrict $\PP$ to a certain residue class $\PP_{a,b}$, where $(a,b)=1$. It seems there are some typos in the original statement and proof of Corollary 2.2 and 2.3 in \cite{BKMST}. For the sake of completeness, we prove the following version of Corollary 2.3 in \cite{BKMST}.

\begin{cor}\label{cor4}
Let $0 <\theta_1 <\theta_2 < \cdots< \theta_m$ and let $\gamma_1, \gamma_2, \ldots, \gamma_m$ be nonzero real numbers such that  $\gamma_j \not \in \mathbb{Q}$ if $\theta_j \in \N$. Then for any $h \in \Z$ and any coprime positive integers $a,b$,  the sequence $$\bigg(\big(\gamma_1 (p-h)^{\theta_1},\gamma_2 (p-h)^{\theta_2}, \ldots, \gamma_m (p-h)^{\theta_m}\big)\bigg)_{p \in \PP_{a,b}}$$ is equdistributed modulo 1 in $\mathbb{T}^m$.
\end{cor}

\begin{proof}
By the multidimensional Weyl's criterion (see for example Section 1.6 of \cite{KN}), it suffices to show that for each
$(\beta_1, \beta_2, \ldots, \beta_m) \in \Z^m \setminus \{(0,0, \ldots, 0)\}$,
\begin{equation}
\sum_{\substack{p \leq x\\ p \equiv b \pmod a}} e\bigg(\sum_{j=1}^m \beta_j\gamma_j (p-h)^{\theta_j} \bigg)=o\bigg(\frac{\pi(x)}{\phi(a)}\bigg)=o\big(\pi(x)\big), \text{ as } x \to \infty.
\end{equation}
By orthogonality relations, 
$$
\frac{1}{a}\sum_{i=1}^a e\bigg(\frac{i(p-b)}{a}\bigg)
=\begin{cases}
1,& \quad p \equiv b \pmod a\\
0, & \quad\text{otherwise}.
\end{cases}
$$
It follows that
\begin{align*}
\sum_{\substack{p \leq x\\ p \equiv b \pmod a}} e\bigg(\sum_{j=1}^m \beta_j\gamma_j (p-h)^{\theta_j} \bigg)
&= \sum_{\substack{p \leq x}} e\bigg(\sum_{j=1}^m \beta_j\gamma_j (p-h)^{\theta_j} \bigg)\frac{1}{a}\sum_{i=1}^a e\bigg(\frac{i(p-b)}{a}\bigg)\\
&=\frac{1}{a} \sum_{i=1}^a  e\bigg(\frac{i(h-b)}{a}\bigg) \sum_{\substack{p \leq x}} e\bigg(\sum_{j=1}^m \beta_j\gamma_j (p-h)^{\theta_j} +\frac{i(p-h)}{a}\bigg).
\end{align*}
Note that for each $1 \leq i \leq a$, $$\xi_i(x):=\sum_{j=1}^m \beta_j\gamma_j x^{\theta_j} +\frac{i}{a}x=\sum_{\beta_j \neq 0} \beta_j\gamma_j x^{\theta_j} +\frac{i}{a}x$$ is of the required form in Theorem \ref{t11} since $\beta_j \neq 0$ and $\theta_j \in \N$ imply $\beta_j \gamma_j \not \in \mathbb{Q}$. Therefore, $\big(\xi_i(p-h)\big)_{p \in \PP}$ is equidistributed modulo 1. By Lemma \ref{Weyl} with $t=1$, as $x \to \infty$, 
\[
 \sum_{i=1}^a  e\bigg(\frac{i(h-b)}{a}\bigg) \sum_{\substack{p \leq x}} e\bigg(\sum_{j=1}^m \beta_j\gamma_j (p-h)^{\theta_j} +\frac{i(p-h)}{a}\bigg)
= \sum_{i=1}^a  e\bigg(\frac{i(h-b)}{a}\bigg) o\big(\pi(x)\big)=o\big(\pi(x)\big). \qedhere
\]
\end{proof}

In particular, for any positive integer $r$, and any coprime positive integers $a,b$, Corollary \ref{cor4} implies that $(p^{r-1/2})_{p \in \PP_{a,b}}$ is equidistributed  modulo 1. Recall $\QQ_{r,d}=\{p \in \PP: p^{2r+1} \equiv 1 \pmod {2d}\}$. It is clear that $\QQ_{r,d}$ is a union of primes in disjoint residue classes:
$$\QQ_{r,d}=\bigcup_{\substack{1 \leq b<2d \\ b^{2r+1} \equiv 1 \pmod{2d}}} \PP_{2d,b}.$$
Using Weyl's criterion, it is easy to show that the union of finitely many disjoint equidistributed sequences is also an equidistributed sequence. Therefore, we obtain the following corollary.

\begin{cor}\label{cor2}
For any positive integers $r$ and $d$, the sequence $(p^{r-1/2})_{p \in \QQ_{r,d}}$ is equdistributed modulo 1.
\end{cor}

\subsection{Proof of Theorem \ref{t12}} In this subsection, we will prove Theorem \ref{t12}. We will use the observation outlined in Section 1.4, which connects the clique number and the number of directions.

By Theorem \ref{t2}, we can deduce the following information about the clique number.
\begin{thm}\label{t9}
Let $q=p^{2r+1} \equiv 1 \pmod {2d}$ such that $r \geq 1$ and $d \geq 3$. Then for any $0<c<(p-1)/2$, the clique number $N=\omega\big(GP(q,d)\big)$ of the generalized Paley graph $GP(q,d)$ satisfies one of the following:
\begin{enumerate}
    \item $N \leq \sqrt{q}-c$.
    \item One of $N,N+1,\ldots, N+\lfloor 2c+\frac{c^2+2c}{\sqrt{q}-c-1} \rfloor$ is a multiple of $p$. 
\end{enumerate}
%In particular, $N \leq \sqrt{q}-1$ when none of $\lfloor \sqrt{q} \rfloor,\lfloor \sqrt{q} \rfloor+1, \lfloor \sqrt{q} \rfloor+2$ is a multiple of $p$.
\end{thm}

\begin{proof}
Lemma \ref{tt} gives the trivial upper bound $N \leq \sqrt{q}$. Since $q$ is not a square, we have $N<\sqrt{q}$. Suppose $N>\sqrt{q}-c$. Then $0<k=q-N^2< q-(\sqrt{q}-c)^2= 2c \sqrt{q}-c^2$ and 
$$
\frac{k}{N-1} <\frac{ 2c \sqrt{q} -c^2}{\sqrt{q}-c-1}=2c+\frac{c^2+2c}{\sqrt{q}-c-1}.
$$
Let $C$ be a clique in $GP(q,d)$ with $|C|=N$. If none of $N,N+1,\ldots, N+\lfloor 2c+\frac{c^2+2c}{\sqrt{q}-c-1} \rfloor$ is a multiple of $p$, then by Theorem \ref{t2}, the number of directions determined by the Cartesian product $C \times C \subset AG(2,q)$ is at least $N^2-N+2$. 
However, each direction formed by  $C \times C$ is a $d$-th power in $\F_q$ or $\infty$, so the number of directions is at most $\frac{q-1}{d}+2$ and we have $ N^2-N+2 \leq \frac{q-1}{d}+2$, i.e. $N(N-1) \leq \frac{q-1}{d}$, or $N \leq \sqrt{\frac{q-1}{d}+\frac{1}{4}}+\frac{1}{2}$. This implies $$\sqrt{q}-\frac{p-1}{2}<\sqrt{q}-c<N \leq  \sqrt{\frac{q-1}{d}+\frac{1}{4}}+\frac{1}{2}<\sqrt{\frac{q}{d}}+1 \leq \sqrt{\frac{q}{3}}+1,$$
that is,  $\sqrt{q}-\sqrt{\frac{q}{3}} \leq \frac{p+1}{2}$, which fails since $q \geq 27$. So one of $N,N+1,\ldots, N+\lfloor 2c+\frac{c^2+2c}{\sqrt{q}-c-1} \rfloor$ must be a multiple of $p$. 
\end{proof}

Note that we assume $c<(p-1)/2$ so that the second condition does not hold automatically. We remark that a similar proof for Theorem \ref{t9} also holds for square $q$, but note that we will only be able to conclude that the clique number is at most $\sqrt{q}$, since $p \mid \sqrt{q}$.

Now we are ready to use Corollary \ref{cor2} and Theorem \ref{t9} to prove Theorem~\ref{t12}. 

\begin{proof}[Proof of Theorem~\ref{t12}]
Since $h(x)=o(x)$ as $x \to \infty$, there is $M>0$ such that $h(x)<\frac{x-1}{2}$ for any $x>M$. Let $X=\{p \in \QQ_{r,d} : \omega\big(P(p^{2r+1},d)\big)>p^{r+1/2}-h(p)\}$. If $X \cap (M, \infty)= \emptyset$, then the statement follows trivially. 

Next we assume $X \cap (M, \infty) \neq \emptyset$. Let $p \in X \cap (M, \infty)$, $q=p^{2r+1}$, and $N=\omega\big(P(q,d)\big)$. Since $N>\sqrt{q}-h(p)$ and $h(p)<\frac{p-1}{2}$, by Theorem \ref{t9}, one of $N,N+1,\cdots, N+\big\lfloor 2h(p)+\frac{h(p)^2+2h(p)}{\sqrt{q}-h(p)-1} \big\rfloor$ is a multiple of $p$. Since $\sqrt{q}-h(p) <N \leq \sqrt{q}$, one of $\lceil \sqrt{q}-h(p) \rceil ,\lceil \sqrt{q}-h(p) \rceil  +1, \ldots, \big \lfloor \sqrt{q}+2h(p)+\frac{h(p)^2+2h(p)}{\sqrt{q}-h(p)-1} \big \rfloor$ must be a multiple of $p$. Therefore, $\lfloor \sqrt{q} \rfloor$ is congruent to one of
$$
 \bigg\lfloor {-}2h(p)-\frac{h(p)^2+2h(p)}{\sqrt{q}-h(p)-1} \bigg\rfloor , \bigg\lfloor {-}2h(p)-\frac{h(p)^2+2h(p)}{\sqrt{q}-h(p)-1} \bigg\rfloor+1, \ldots, \lceil h(p) \rceil \pmod p.
$$
Note that $\sqrt{q}=p^{r+1/2}$. If $0 \leq m <p$, then $\lfloor \sqrt{q} \rfloor \equiv \lfloor p\{ p^{r-1/2} \} \rfloor \equiv m \pmod p$ is equivalent to 
$
\{ p^{r-1/2} \} \in [\frac{m}{p},\frac{m+1}{p}).
$
Therefore, $p \in X \cap (M, \infty)$ implies that 
$$
\{ p^{r-1/2} \} \in \bigg[0,\frac{\lceil h(p)\rceil +1}{p}\bigg) \cup \bigg[1-\frac{\big\lfloor {-}2h(p)-\frac{h(p)^2+2h(p)}{\sqrt{q}-h(p)-1} \big\rfloor}{p},1\bigg).
$$
Since $h(x)=o(x)$ as $x \to \infty$,
$$2h(x)+\frac{h(x)^2+2h(x)}{x^{r+1/2}-h(x)-1}=o(x)+o\big(x^{3/2-r}\big)=o(x) \quad \text{ as } x \to \infty.$$
Then for any $\varepsilon>0$, there exists $M_\varepsilon>M$ such that $\{ p^{r-1/2} \} \in [0,\varepsilon) \cup [1-\varepsilon,1)$ for any $p \in X \cap (M_\varepsilon, \infty)$. Therefore, for any $\varepsilon>0$, by the equidistribution of $(p^{r-1/2})_{p \in \QQ_{r,d}}$, the relative upper density of $X \subset \QQ_{r,d}$ is at most $2\varepsilon$. Letting $\varepsilon \to 0^+$, we conclude that the relative density of $X \subset \QQ_{r,d}$ is zero. Therefore, $\omega\big(P(p^{2r+1},d)\big)\leq p^{r+1/2}-h(p)$ holds for almost all $p \in \QQ_{r,d}$.
\end{proof}

\section*{Acknowledgement}
The author would like to thank Greg Martin, J\'ozsef Solymosi, Ethan White, and Joshua Zahl for valuable suggestions. The author would also like to thank Daniel Di Benedetto, Gabriel Currier, and Stephanie van Willigenburg for helpful discussions.

\end{document}